\documentclass[11pt]{amsart}
\usepackage{amsmath,amssymb,euscript,mathrsfs}
\usepackage[small,heads=LaTeX,nohug]{diagrams}
\usepackage[colorlinks=true,linkcolor=black,citecolor=black,urlcolor=black]{hyperref}
\usepackage{enumerate}

\pushQED{\qed}

\newcommand{\define}{\textbf}

\newcommand{\isom}{\cong}
\renewcommand{\setminus}{\smallsetminus}
\renewcommand{\phi}{\varphi}

\renewcommand{\tilde}{\widetilde}

\newcommand{\C}{\mathbb{C}}
\newcommand{\Q}{\mathbb{Q}}

\newcommand{\N}{\mathbb{N}}
\newcommand{\Z}{\mathbb{Z}}
\renewcommand{\P}{\mathbb{P}}
\newcommand{\A}{\mathbb{A}}

\newcommand{\E}{\mathbb{E}}

\newcommand{\GG}{\mathcal{G}}
\newcommand{\HH}{\mathcal{H}}

\newcommand{\liep}{\mathfrak{p}}
\newcommand{\lieb}{\mathfrak{b}}

\newcommand{\J}{J}

\newcommand{\mJ}{\overline{J}^\circ}

\newcommand{\Fl}{Fl}
\newcommand{\BH}{\overline{H}}
\newcommand{\orb}{\mathrm{orb}}
\newcommand{\sing}{\mathrm{Sing}}
\newcommand{\sm}{\mathrm{Sm}}
\newcommand{\pt}{\mathrm{pt}}

\newcommand{\an}{\mathrm{an}}

\newcommand{\mm}{\mathbf{m}}
\newcommand{\ab}{\mathbf{a}}
\newcommand{\rr}{\mathbf{r}}

\newcommand{\te}{\tilde{e}}

\DeclareMathOperator{\Sym}{Sym}
\DeclareMathOperator{\codim}{codim}
\DeclareMathOperator{\rk}{rk}

\DeclareMathOperator{\Hom}{Hom}
\DeclareMathOperator{\Spec}{Spec}
\DeclareMathOperator{\ord}{ord}
\DeclareMathOperator{\Cont}{Cont}
\DeclareMathOperator{\id}{id}

\DeclareMathOperator{\Bl}{Bl}
\DeclareMathOperator{\lct}{lct}
\DeclareMathOperator{\SR}{SR}

\newtheorem{theorem}{Theorem}[section]
\newtheorem{lemma}[theorem]{Lemma}
\newtheorem{proposition}[theorem]{Proposition}
\newtheorem{corollary}[theorem]{Corollary}
\newtheorem{conjecture}[theorem]{Conjecture}

\theoremstyle{definition}
\newtheorem{definition}[theorem]{Definition}
\newtheorem{remark}[theorem]{Remark}
\newtheorem{example}[theorem]{Example}
\newtheorem{question}[theorem]{Question}

\newcommand{\excise}[1]{}

\begin{document}

\title[Arc spaces and equivariant cohomology]{Arc spaces and equivariant cohomology}
\author{Dave Anderson}
\address{Department of Mathematics\\University of Washington\\Seattle, WA 98195}
\email{dandersn@math.washington.edu}
\author{Alan Stapledon}
\address{Department of Mathematics\\University of British Columbia\\ BC, Canada V6T 1Z2}
\email{astapldn@math.ubc.ca}

\keywords{}

\date{August 3, 2011}
\thanks{This work was initiated and partially completed 
while the authors were graduate students 
at the University of Michigan.  
A.S. was funded in part with assistance from the Australian 
Research Council project DP0559325 at Sydney University, Chief Investigator Professor G I Lehrer.  
Part of this work was completed during the period D.A. was employed by the Clay Mathematics Institute as a Liftoff Fellow.  D.A. was also partially supported by NSF Grants DMS-0502170 and DMS-0902967.}

\subjclass[2000]{55N91, 14E99.}

\begin{abstract}
We present a new geometric interpretation of equivariant cohomology in which one replaces a smooth, complex $G$-variety $X$ by its associated arc space $J_{\infty} X$, with its induced $G$-action.  This not only allows us to obtain geometric classes in equivariant cohomology of arbitrarily high degree, but also provides more flexibility for equivariantly deforming classes and geometrically interpreting multiplication in the equivariant cohomology ring.  Under appropriate hypotheses, we obtain explicit bijections between $\Z$-bases for the equivariant cohomology rings of smooth varieties related by an equivariant, proper birational map.  We also show that self-intersection classes can be represented as classes of contact loci, under certain restrictions on singularities of subvarieties.

We give several applications.  Motivated by the relation between self-intersection and contact loci, we define higher-order equivariant multiplicities, generalizing the equivariant multiplicities of Brion and Rossmann; these are shown to be local singularity invariants, and computed in some cases.  We also present geometric $\Z$-bases for the equivariant cohomology rings of a smooth toric variety (with respect to the dense torus) and a partial flag variety (with respect to the general linear group).
\end{abstract}

\maketitle
\tableofcontents

\section{Introduction}

Let $X$ be a smooth complex algebraic variety equipped with an action of a linear algebraic group $G$.  In this article, we consider two constructions associated to this situation.  The \emph{equivariant cohomology ring} $H_G^*X$ is an interesting and useful object encoding information about the topology of $X$ as it interacts with the group action; for example, fixed points and orbits are relevant, as are representations of $G$ on tangent spaces.  The \emph{arc space} of $X$ is the scheme $J_\infty X$ parametrizing morphisms $\Spec\C[[t]] \to X$; this construction is functorial, so $J_\infty G$ is a group acting on $J_\infty X$.  Except when $X$ is zero-dimensional, $J_\infty X$ is not of finite type over $\C$, but it is a pro-variety, topologized as a certain inverse limit.  Due to their connections with singularity theory \cite{ELMContact, EMInversion, MusJet} and their central role in motivic integration \cite{DLGerms, KonMotivic}, arc spaces have recently proved increasingly useful in birational geometry.  

The present work stems from a simple observation: The projection $J_\infty X \to X$ is a homotopy equivalence, and is equivariant with respect to $J_\infty G \to G$, so there is a canonical isomorphism $H_G^*X = H_{J_\infty G}^*J_\infty X$ (Lemma~\ref{l:homotopy}).  Very broadly, our view is that interesting classes in $H_G^*X$ arise from the $J_\infty G$-equivariant geometry of $J_\infty X$.  The purpose of this article is to initiate an investigation of the interplay between information encoded in $H_G^*X$ and in $J_\infty X$.

The philosophy we wish to emphasize is motivated by analogy with two notions from ordinary cohomology of (smooth) algebraic varieties.  First, interesting classes in $H^{2k}X$ come from subvarieties of codimension $k$.  We seek invariant subvarieties of codimension $k$ to correspond to classes in $H_G^{2k}X$.  Since $H_G^*X$ typically has nonzero classes in arbitrarily large degrees, however, $X$ must be replaced with a larger---in fact, infinite-dimensional---space.  The traditional approach to equivariant cohomology, going back to Borel, replaces $X$ with the \emph{mixing space} $\mathbb{E}G \times^G X$, which does not have a $G$-action; we will instead study $J_\infty X$, which is intrinsic to $X$ and on which $G$ acts naturally.

The second general notion is that cup product in $H^*X$ should correspond to transverse intersection of subvarieties.  In the $C^\infty$ category, of course, this is precise: any two subvarieties can be deformed to intersect transversely, and the cup product is represented by the intersection.  On the other hand, $X$ often has only finitely many $G$-invariant subvarieties, so in the equivariant setting, no such moving is possible within $X$ itself.  Replacing $X$ with $J_\infty X$, one gains much greater flexibility to move invariant cycles.

A new and remarkable feature of this approach is that in the case when $G$ is the trivial group, one obtains interesting results for both the ordinary cohomology ring of $X$ and the geometry of the arc spaces of subvarieties of $X$.

Our first main theorem addresses the first notion, and says that under appropriate hypotheses, $J_\infty G$-orbits in $J_\infty X$ determine a basis over the integers:

\bigskip \noindent {\bf Theorem~\ref{t:basis}.} {\em
Let $G$ be a connected linear algebraic group acting on a smooth complex algebraic variety $X$, with $D\subseteq X$ a $G$-invariant closed subset such that $G$ acts on $X\setminus D$ with unipotent stabilizers.  Suppose $J_\infty X \setminus J_\infty D = \bigcup_j U_j$ is an \emph{equivariant affine paving}, in the sense of Definition~\ref{d:paving}.  Then 
\[
  H_G^*X = \bigoplus_j \Z\cdot [\overline{U}_j].
\]
}
\bigskip

\noindent
When $G$ is a torus, the condition that $G$ act on $X\setminus D$ with unipotent stabilizers is a ``generic freeness'' hypothesis, and is automatic in many cases of interest.  (When $G$ is trivial, this reduces to a well-known fact about affine pavings; see, e.g., \cite[Appendix B, Lemma 6]{FulYT}.)

Since arc spaces are well-suited to the study of the birational geometry of $X$, one should also consider proper equivariant birational maps $f\colon Y \to X$.  When $X$ satisfies the conditions of Theorem~\ref{t:basis} and the paving is \emph{compatible} with $f$, we establish a geometric bijection between $\Z$-bases of $H_G^*X$ and $H_G^*Y$ (Corollary~\ref{c:birational}).  This result is new even in the case when $G$ is trivial.



In Sections~\ref{s:mult1}~and~\ref{s:mult2}, we address the second notion, and relate the cup product in $H_G^*X$ to intersections in $J_\infty X$.  The main results of these sections 
(Theorem~\ref{t:multiplication} and Theorem~\ref{t:smooth}) 
say that under certain restrictions on the singularities of $G$-invariant subvarieties, products of their equivariant cohomology classes are represented by \emph{multi-contact loci} in the arc space of $X$.  (Basic facts about contact loci are reviewed in \S\ref{arcjet}.)  

Throughout this article, we make frequent use of the \emph{jet schemes}
\[
  J_m V = \Hom(\Spec \C[t]/(t^{m + 1}), V),
\]
which may be considered finite-dimensional approximations to the arc space $J_\infty V$.
A key special case of our theorems about multiplication says that given a $G$-invariant subvariety $V\subseteq X$, provided its singularities are sufficiently mild, we have an equality
\begin{equation}\label{eq:magic}
  [V]^{m+1} = [J_mV]
\end{equation}
as classes in $H_G^*X = H_{J_mG}^*(J_mX)$.  This special case may be summarized a little more precisely as follows:

\bigskip \noindent {\bf Corollary~(of Theorems~\ref{t:multiplication} and \ref{t:smooth}).} {\em
Let $G$ be a connected reductive group, and fix an integer $m>0$.  
\renewcommand{\theenumi}{\alph{enumi}}
\begin{enumerate}
\item Assume $X^G$ is finite, the natural map $\iota^*\colon H_G^*X \to H_G^*X^G$ is injective, and $V\subseteq X$ 
is an \emph{equivariant local complete intersection} (see \S\ref{s:mult1}), with $\codim(J_mV,J_mX) = (m+1)c$.  Then $[V]^{m+1} = [J_mV]$. \smallskip \label{mult-part1}
\item Assume $V \subseteq X$ is a connected $G$-invariant subvariety of codimension $c$, with $\codim(\sing(V),X)>(m+1)c$.  
Then $[V]^{m+1} = [\overline{J_m\sm(V)}]$. \label{mult-part2}
\end{enumerate}
}
\bigskip

\noindent
When $G$ is a torus, the assumptions on the fixed locus in Part~\eqref{mult-part1} are part of a standard package of hypotheses for localization theorems in equivariant cohomology; for more general groups, see Remark~\ref{r:q-coeff}.  Note that the statement in Part~\eqref{mult-part2} applies in particular to any smooth subvariety $V\subseteq X$.

Information about the singularities of $V$ is encoded in the geometry of its jet schemes, but these spaces are notoriously difficult to compute.  In fact, almost nothing is known about them, except when $V$ is a local complete intersection \cite{EMInversion, MusJet, MusSingularities}---in which case the sequence of dimensions $\{ \dim J_m V \}_{m \ge 0}$ determines the log canonical threshold of $V$ \cite[Corollary 0.2]{MusSingularities}---or when $V$ is a determinantal variety \cite{KSDeterminantal, YueJet}.  In particular, the corresponding class $[J_m V] \in H^*_G X$ is an important invariant.  When $X = \A^d$ and $G = (\C^*)^r$, this class is the \emph{multi-degree} of $J_m V$ (see Remark~\ref{r:multi}, \cite[Chapter~8]{MSCombinatorial}).



Self-intersection is perhaps the most difficult part of intersection theory to interpret geometrically, because it requires some version of a moving lemma.  From an intersection-theoretic point of view, Equation~\eqref{eq:magic} gives a new geometric interpretation of the self-intersection $[V]^{m + 1}$, even in the case when $G$ is trivial and $V$ is smooth.  From the perspective of singularity theory, this gives a highly non-trivial calculation of the class $[J_m V]$ under suitable conditions.



The above Corollary implies a relationship between the failure of Equation~\eqref{eq:magic} and the singularities of $V$.  To measure this discrepancy, in \S\ref{s:multiplicities} we introduce \emph{higher-order equivariant multiplicities}.  These generalize the equivariant multiplicities considered by Rossmann \cite{rossmann} and Brion \cite{BriChow}, among others; the latter have been used to study singularities of Schubert varieties \cite[\S6.5]{BriChow}, and are related to Minkowski weights on fans \cite{katz-payne}.  We prove the higher-order multiplicities are intrinsic to $V$ (Theorem~\ref{t:mult-invt}), and apply our main results to relate them with the ($0^\mathrm{th}$-order) multiplicities of Brion.

Our initial motivation for this work came from the theory of toric varieties.  By a theorem of Ishii, orbits of generic arcs in a toric variety $X$ are parametrized by the same set which naturally indexes a $\Z$-basis for $H_T^*X$, namely, points in the lattice $N$ of one-parameter subgroups of $T$.  In \S\ref{s:toric}, we give a geometric interpretation of this bijection by applying our results to extend it to an isomorphism of rings (Corollary~\ref{tv}), reproving the well-known fact that the equivariant cohomology of a smooth toric variety is isomorphic to the Stanley-Reisner ring of the corresponding fan.  We expect this intriguing picture to extend to a relation between \emph{equivariant orbifold cohomology} of toric stacks and spaces of \emph{twisted arcs} in the sense of Yasuda \cite{YasMotivic} (see \S\ref{rmks}).

As another application, we consider the action of $GL_n$ on $n\times n$ matrices by left multiplication.  In \S\ref{s:gln}, we show that Theorem~\ref{t:basis} applies to this situation, using a paving defined in terms of contact loci with certain determinantal varieties (Corollary~\ref{c:glnbasis}).  Using our results concerning the behavior of equivariant cohomology under birational maps (Corollary~\ref{c:birational}), we then deduce an arc-theoretic basis for the $GL_n$-equivariant cohomology of a partial flag variety (Corollary~\ref{c:flag}).

In the case when $V$ is a determinantal variety cut out by maximal minors, Ko{\v{s}}ir and Sethuraman proved that the jet schemes $J_m V$ are irreducible \cite[Theorem~3.1]{KSDeterminantal}.  The hypotheses for Theorems~\ref{t:multiplication}~and~\ref{t:smooth} fail for these subvarieties---determinantal varieties are generally not l.c.i., and they have large singular sets---but the conclusions appear to hold (Conjecture~\ref{c:gln}); it would be interesting to have a more general framework which explains this.  An intriguing consequence of Equation~\eqref{eq:magic} and the more general Theorems~\ref{t:multiplication}~and~\ref{t:smooth} is that they allow us to conjecture, and prove in some cases, formulas for the multi-degrees of $J_m V$ (Conjecture~\ref{c:gln}, Remark~\ref{r:known})---a calculation which Macaulay~2 can perform in very few examples.

In the theory of equivariant cohomology, one often chooses finite-di{\-}men{\-}sional algebraic approximations to the mixing space; see, e.g., \cite[\S2]{FulEq}.  
(This approach was used by Totaro, and further developed by Edidin and Graham, to define an algebraic theory of {\it equivariant Chow groups}.)  
In this context, one may attempt to find representatives for classes in $H_G^*X$ via subvarieties of the approximation space (cf. \cite[\S2.2]{BriChow}) or deform to transverse position to compute products (cf. \cite{AndPos}).  As mentioned above, our approach uses the jet schemes $J_m X$ as finite-dimensional approximations to $J_\infty X$.  These seem to be unrelated to the mixing space approximations; as with the arc space, they have the advantages of being intrinsic to $X$ and carrying large group actions.

Equivariant classes in jet schemes have also been studied by B\'{e}rczi and Szenes \cite{BS}, from a somewhat different point of view.  Our results overlap in a simple special case.  They consider the space
\[
  J_d(n,k) = \Hom(\Spec \C[t_1,\ldots,t_n]/(t_1,\ldots,t_n)^{d+1}, \A^k) ,
\]
and compute the classes of contact loci $\Cont^d(\{0\})$.  In general, this is quite complicated, but in our case, when $n=1$, the class in question is $c_k^d \in H_{GL_k}^*J_d(1,k) = \Z[c_1,\ldots,c_k]$.  This is also an easy case of Conjecture~\ref{c:gln} (see Remark~\ref{r:known}(\ref{r:Vn})).

Arc spaces have also been used by Arkhipov and Kapranov to study the \emph{quantum cohomology} of 
toric varieties \cite{AKQuantum}.  There may be an interesting relation between their point of view 
and ours, but we do not know a direct connection.  

For the convenience of the reader, we include brief summaries of basic facts about equivariant cohomology (\S\ref{s:eq-coh}) and jet schemes (\S\ref{arcjet}), together with references.  In \S\ref{equivgeom}, we prove a technical fact about stabilizers (Proposition~\ref{p:boundary}) which is used in the proof of Theorem~\ref{t:basis}.  The main results and applications described above are contained in \S\S\ref{jet&eq}--\ref{s:gln}.  We conclude the paper with a short discussion of questions and projects suggested by the ideas presented here.

\medskip
\noindent
{\it Notation and conventions.}  
All schemes are over the complex numbers.  For us, a \define{variety} is a separated reduced scheme of finite type over $\C$, assumed to be pure-dimensional but not necessarily irreducible.  Throughout, $G$ will be a connected linear algebraic group over $\C$, and $X$ will be a $G$-variety.

Unless otherwise indicated, cohomology will be taken with $\Z$ coefficients, with respect to the usual (complex) topology.

\medskip
\noindent
{\it Acknowledgements.}  The authors would like to thank Mircea Musta\c t\v a for several enlightening discussions, and Mark Haiman for suggesting that the calculation at the end of \S\ref{s:gln} should generalize from projective space to all partial flag varieties.  We also thank Sara Billey, Bill Fulton, and Rich\'ard Rim\'anyi for helpful comments.

\section{Equivariant cohomology}\label{s:eq-coh}

We refer the reader to \cite{FulEq} or \cite{BriEq} for an introduction to equivariant cohomology, as well as proofs and details.  Here we collect the basic properties we will need, and give a few illustrative examples.  As always, $G$ is a connected linear algebraic group acting on the left on $X$.\footnote{
By definition, $H_G^*X$ is the singular cohomology of the Borel mixing space $\E G\times^G X$; equivalently, it is the cohomology of the quotient stack $[G\backslash X]$.  The reader may consult one of the above references for a discussion of this construction.}

A map $f\colon X\to X'$ is \define{equivariant} with respect to a homomorphism $\phi\colon G \to G'$ if $f(g\cdot x) = \phi(g)\cdot f(x)$ for all $g\in G$, $x\in X$.  Equivariant cohomology is contravariant for equivariant maps: one has $f^*\colon H_{G'}^*X' \to H_G^*X$.

The following two facts play a key role in our arguments:

\begin{lemma}\label{l:homotopy}
Suppose $X\to X'$ is equivariant with respect to $G \to G'$, and suppose both maps induce (weak) homotopy equivalences.  Then the induced map $H_{G'}^*X' \to H_G^*X$ is an isomorphism.
\end{lemma}

\noindent
(In most of our applications of Lemma~\ref{l:homotopy}, both maps will be locally trivial fiber bundles with contractible fibers---here both the hypothesis and conclusion are easily verified.)

\begin{lemma}\label{l:orbit}
The equivariant cohomology of an orbit is described as follows: for a closed subgroup $G'\subseteq G$, one has
\[
  H_G^*(G/G') = H_{G'}^*(\pt).
\]
\end{lemma}

\begin{example}\label{e:representation}
For a representation $V$ of $G$, one has $H_G^*V = H_G^*(\pt)$. 
\end{example}

\excise{
Similarly, if $B\subseteq GL_n$ is the Borel subgroup of upper-triangular matrices, one has $H_T^*X = H_B^*X$ for any $B$-space $X$.}

\begin{example}\label{ex:contractible}
If $G$ is contractible, then $H_G^*(\pt) = \Z$.  
\end{example}

When $X$ is smooth, a closed $G$-invariant subvariety $Z\subseteq X$ of codimension $c$ defines a class $[Z]$ in $H_G^{2c}X$.  If $Z_1,\ldots,Z_k$ denote the irreducible components of $Z$, then $[Z] = [Z_1] + \cdots + [Z_k]$.

An equivariant vector bundle $V \to X$ has \define{equivariant Chern classes} $c^G_i(V)$ in $H_G^{2i}X$, with the usual functorial properties of Chern classes.

\begin{example}\label{ex:gln}
An equivariant vector bundle on a point is simply a representation of $G$, so one has corresponding Chern classes $c^G_i(V) \in H_G^*(\pt)$.  For $V=\C^n$, with $GL_n$ acting by the standard representation, the Chern classes $c_i = c^G_i(V)$ freely generate $H_{GL_n}^*(\pt)$.
\end{example}

A key feature of equivariant cohomology is that $H_G^*X$ is canonically an algebra over $H_G^*(\pt)$, via the constant map $X \to \pt$.  In contrast to the non-equivariant situation, $H_G^*(\pt)$ is typically not trivial.  

\begin{example}
If $T\isom(\C^*)^n$ is a torus with character group $M\isom\Z^n$, then $H_T^*(\pt) = \Sym^* M \isom \Z[t_1,\ldots,t_n]$.  The inclusion $(\C^*)^n \hookrightarrow GL_n$ induces an inclusion
\[
  H_{GL_n}^*(\pt) = \Z[c_1,\ldots,c_n] \hookrightarrow \Z[t_1,\ldots,t_n],
\]
sending $c_i$ to the $i$th elementary symmetric function in $t$.
\end{example}

We will use \define{equivariant Borel-Moore homology} $\BH^G_*X$ as a technical tool; see \cite[p.605]{EGEIT} or \cite[Section 1]{BriPoincare} for some details.  The main facts are analogous to the non-equivariant case, for which a good reference is \cite[Appendix B]{FulYT}; we summarize them here.

If $X$ has (pure) dimension $d$, then $\BH^G_i X = 0$ for $i>2d$ and $\BH^G_{2d}X = \bigoplus \Z$, with one summand for each irreducible component of $X$.  In contrast to the non-equivariant case, $\BH^G_i X$ may be nonzero for arbitrarily negative $i$. If $X$ is smooth of dimension $d$, then $\BH^G_i X = H_G^{2d-i}X$.

Borel-Moore homology is covariant for equivariant proper maps and contravariant for equivariant open inclusions.  For $Z\subseteq X$ a $G$-invariant closed subvariety of codimension $c$, there is a fundamental class $[Z]$ in $\BH^G_{2d-2c}X$.  More generally, if $Z\subseteq X$ is any $G$-invariant closed subset, with $U=X\setminus Z$ the open complement, there is a long exact sequence
\[
  \cdots \to \BH^G_i Z \to \BH^G_i X \to \BH^G_i U \to \BH^G_{i-1} Z \to \cdots.
\]

\begin{definition}\label{def:trivialBM}
A $d$-dimensional variety $X$ has \define{trivial equivariant Borel-Moore homology} if
\[
  \BH^G_i X = \left\{\begin{array}{cl} \Z & \text{if } i=2d ; \\ 0 & \text{otherwise}. \end{array}\right.
\]
\end{definition}

\begin{example}\label{ex:trivialBM}
For us, the main examples of such varieties arise as follows.  An \define{affine family of 
$G$-orbits} is a smooth map $S \to \A^n$ of $G$-varieties, with $G$ acting trivially on $\A^n$, 
such that there is a section $s\colon\A^n \to S$, and the map $G \times \A^n \to S$, 
$(g,x)\mapsto g\cdot s(x)$ is smooth and surjective.  In other words, as a smooth scheme over 
$\A^n$, $S$ is the geometric quotient of the group scheme $\GG = G \times \A^n$ by a closed 
subgroup scheme $\HH$ over $\A^n$, so we may write $S = \GG/\HH$.

When $\HH \to \A^n$ has contractible fibers---i.e., the stabilizers (in $G$) of points in $S$ are contractible subgroups---the projection $S \to \A^n$ is a (Serre) fibration, by \cite[Corollary 15(ii)]{MeiSubmersions}.  It follows that $H_G^*(S) = H_G^*(G/H_0) = \Z$, where $G/H_0 \subseteq S$ is the fiber over $0\in\A^n$.  Since $S$ is smooth, we conclude that $S$ has trivial Borel-Moore homology.
\end{example}

The following is an equivariant analogue of \cite[Appendix B, Lemma 6]{FulYT}:

\begin{lemma}\label{l:basis}
Suppose $X$ has a filtration by $G$-invariant closed subvarieties $X_{s} \subseteq X_{s-1} \subseteq \cdots \subseteq X_0 = X$ such that each complement $U_{i}=X_i\setminus X_{i+1}$ has trivial equivariant Borel-Moore homology.  Then, for $0\leq k < \codim(X_{s},X)$, we have
\[
  \BH^G_{2d-2k}X = \bigoplus_{\codim U_i = k} \Z\cdot[\overline{U}_i]
\]
and $\BH^G_{2d-2k+1}X =0$.  Consequently, if $X$ is smooth we have
\[
  H_G^{2k}X = \bigoplus_{\codim U_i = k} \Z\cdot[\overline{U}_i]
\]
and $H_G^{2k-1}X=0$, for $0\leq k < \codim(X_{s},X)$.
\end{lemma}

\noindent
We omit the proof, which proceeds exactly as in the non-equivariant case (using induction and the long exact sequence).

We will also need a slight refinement, whose proof is immediate from the long exact sequence:
\begin{lemma}\label{l:openBM}
Let $X_0 \subseteq X$ be a $G$-invariant open subset.  Then the induced map $\BH^G_k X \to \BH^G_k X_0$ is an isomorphism for $2d \geq k > 2\dim(X\setminus X_0)+1$.
\end{lemma}

\section{Arc spaces and jet schemes}\label{arcjet}

In this section, we review some aspects of the theory of arc spaces and jet schemes, and set notation for the rest of the paper.  We refer the reader to \cite{MusJet} and \cite{EMJet} for more details. 

Let $X$ be a scheme over $\C$ of finite type.
The $m^{\textrm{th}}$ \define{jet scheme} of $X$ is a scheme $J_{m}X$ over $\C$ whose 
$\C$-valued points parameterize all morphisms $\Spec \C[t]/(t^{m + 1}) \rightarrow X$. 
For example, $J_{0}X = X$ and $J_{1}X = TX$ is the total tangent space of $X$. 
In what follows, we will often identify schemes with their $\C$-valued points. 

For $m \geq n$, the natural ring homomorphism $\C[t]/(t^{m + 1}) \rightarrow \C[t]/(t^{n + 1})$ induces truncation morphisms
\[
  \pi_{m,n}\colon J_{m}X \rightarrow J_{n}X,
\]
and we write 
\[
  \pi_m = \pi_{m,0}\colon J_{m}X \rightarrow X.
\]
The inclusion $\C \hookrightarrow \C[t]/(t^{m + 1})$ induces a morphism $\Spec \C[t]/(t^{m + 1}) \rightarrow \Spec \C$, and hence a morphism
\[
  s_m\colon X \rightarrow J_mX,
\]
called the \define{zero section}, with the property that $\pi_m \circ s_m = \id$.

The truncation morphisms $\pi_{m, m-1}\colon J_{m}X \rightarrow J_{m-1}X$ form a projective system whose projective limit is a scheme $J_{\infty}X$ over $\C$, which is typically not of finite type. The scheme $J_{\infty}X$ is called the \define{arc space} of $X$, and the $\C$-valued points of $J_{\infty}X$ parameterize all morphisms $\Spec \C[[t]] \rightarrow X$.  For each $m$, there is a truncation morphism
\[
  \psi_{m}\colon J_{\infty}X \rightarrow J_{m}X,
\]
induced by the natural ring homomorphism  $\C[[t]] \rightarrow \C[[t]]/(t^{m + 1}) = \C[t]/(t^{m + 1})$.

Both $J_{m}$ and $J_{\infty}$ are functors from the category of schemes of finite type over $\C$ to the category of schemes over $\C$, and both preserve fiber squares (cf. \cite[Remark 2.8]{EMJet}).  For a morphism $f\colon X \to Y$, we write $f_m \colon J_m X \rightarrow J_m Y$ for the corresponding morphism of jet schemes.  The following lemma should be compared with Theorem~\ref{lct}.

\begin{lemma}\cite[Proposition 5.12]{EMJet}\label{thin}
If $X$ is a smooth variety and $V$ is a closed subscheme of $X$ with $\dim V < \dim X$, then
\[
 \lim_{m \to \infty}  \codim(J_m V, J_m X) = \infty.
\]
\end{lemma}



The fundamental fact we exploit in this paper is the following:

\begin{lemma}[{\cite[Corollary 2.11]{EMJet}}]\label{smootharc}
If $X$ is a smooth variety of dimension $d$, then $J_mX$ is a smooth variety of dimension 
$(m + 1)d$, and the truncation morphisms $\pi_{m, m- 1}\colon J_{m}X \rightarrow J_{m-1}X$ are Zariski-locally trivial fibrations with fiber $\mathbb{A}^{d}$.  Moreover, the projection $\psi_0\colon J_{\infty} X \rightarrow X$ is a Zariski-locally trivial fibration with contractible fibers.  
\end{lemma}

A little more can be said about the projections, still in the smooth case:
\begin{lemma}[{see \cite[Proposition~2.6]{ish}}]\label{l:reltan}
If $X$ is a smooth variety, the relative tangent bundle for the truncation map $J_{m}X \rightarrow J_{m-1}X$ is isomorphic to $\pi_m^*TX$.
\end{lemma}

When $X$ is singular, $J_m X$ may not be reduced or irreducible, and may not be pure-dimensional.  However, if $\sm(X)$ denotes the smooth locus of $X$, then the closure of $\pi_m^{-1} \sm(X) \subseteq J_m X$ is an irreducible component of dimension $(m + 1)d$.

\begin{example}\label{affine1}
Let $X = \mathbb{A}^{n} = \Spec \C[x_{1}, \ldots, x_{n}]$.  An $m$-jet $\Spec \C[t]/(t^{m + 1}) \rightarrow \A^n$
corresponds to a ring homomorphism $\C[x_{1}, \ldots, x_{n}] \rightarrow \C[t]/(t^{m + 1})$, and hence to an $n$-tuple of polynomials in $t$ of degree at most $m$.  
We conclude that
$J_{m} \A^n \cong \mathbb{A}^{(m + 1)n}$, and we write $\{ x_i^{(j)} \mid 1 \le i  \le r, 0 \le j \le m \}$ for the corresponding coordinates.

Similarly, an arc is determined by an $n$-tuple of power series over $\C$, and 
$J_{\infty} \A^n$ is an infinite-dimensional affine space. 
\end{example}

\begin{example}\label{e:equations}
With the notation of the previous example, if $X \subseteq \A^n$ is defined by equations $\{ f_1(x_1, \ldots, x_n) = \cdots = f_r(x_1, \ldots, x_n) = 0 \}$, then
an $m$-jet $\Spec \C[t]/(t^{m + 1}) \rightarrow X$ corresponds to a ring homomorphism
\[
  \C[x_{1}, \ldots, x_{n}]/(f_1, \ldots, f_r) \rightarrow \C[t]/(t^{m + 1}).
\]
The closed subscheme
$J_m X \subseteq J_m \A^n \cong \A^{(m + 1)n}$ is therefore defined by the equations
\[
  f_i\left(\sum_{j = 0}^m x_1^{(j)} t^j, \ldots, \sum_{j = 0}^m x_n^{(j)} t^j\right) 
  \equiv 0 \mod{t^{m + 1}} \quad \text{ for } 1\leq i\leq r.
\]
In other words, let $f_i^{(k)}$ be the coefficient of $t^k$ in $f_i(\sum_{j = 0}^m x_1^{(j)} t^j, \ldots, \sum_{j = 0}^m x_n^{(j)} t^j)$, so it is a polynomial in the variables $\{ x_i^{(j)} \mid 1 \le i  \le r,\; 0 \le j \le m \}$.  Then $J_m X$ is defined by the $(m + 1)r$ equations 
$\{ f_i^{(k)} = 0 \mid 1 \le i  \le r,\; 0 \le k \le m \}$. 

In fact, if $R =  \C[ x_i^{(k)} \mid 1 \le  i \le n,\; k \ge 0 ]$  and $D: R \rightarrow R$ is the unique derivation over $\C$ satisfying $D( x_i^{(k)} ) = x_i^{(k + 1)}$, then $f_i^{(k)} = D^k (f_i)$ \cite[p. 5]{MusJet}. 
\end{example}

Assume $X$ is smooth of dimension $d$.  A \define{cylinder} $C$  in $J_{\infty}X$ is a subset of the arc space of $X$ of the form $C = \psi_{m}^{-1}(S)$, for some $m \geq 0$ and 
some constructible subset $S \subseteq J_{m}X$.  The cylinder $C$ is called open, closed, locally closed, or irreducible if the corresponding property holds for $S$, and the codimension of $C$ is defined to be the codimension of $S$ in $J_{m}X$.  That these notions are well-defined follows from the fact that $\pi_{m,m-1}$ is a Zariski-locally trivial fibration with fiber $\mathbb{A}^{d}$ (Lemma~\ref{smootharc}).  A subset  of $J_\infty X$ is called \define{thin} if it is contained in $J_\infty V$ for some proper, closed subset $V \subseteq X$. 

\begin{lemma}\cite[Proposition 5.11]{EMJet}\label{l:codimension}
Let $X$ be a smooth variety and let $C \subseteq J_\infty X$ be a cylinder. If the complement of a disjoint union of cylinders $\coprod_j C_j \subseteq C$ is thin, then  $\lim_{j \to \infty} \codim C_j = \infty$ and $\codim C = \min_j \codim C_j$. 
\end{lemma}

Interesting examples of cylinders arise as follows.  Let $V$ be a proper, closed subscheme of $X$ defined by an ideal sheaf $\mathcal{I}_V \subseteq \mathcal{O}_X$, and let
$\gamma\colon \Spec \C[[t]] \rightarrow X$ be an arc.  The pullback 
of $\mathcal{I}_V$ via $\gamma$ is either an ideal of the form $(t^{\alpha})$, for some non-negative integer $\alpha$, or the zero ideal.  In the former case, the \define{contact order} $\ord_{\gamma}(V)$ of $V$ along $\gamma$ is defined to be $\alpha$; in the latter case, $\ord_{\gamma}(V)$ is infinite by convention, and $\gamma$ lies in  $J_{\infty}V \subseteq J_\infty X$.  For each non-negative integer $e$, set 
\[
  \Cont^{\geq e}(V) = \{ \gamma \in J_{\infty}X \mid \ord_{\gamma}(V) \geq e \},
\]
so
$\Cont^{\geq 0}(V) = J_{\infty}X$ and $\Cont^{\geq e}(V) = \psi_{e - 1}^{-1}(J_{e - 1}V)$ for $e > 0$.  We see that $\Cont^{\geq e}(V)$ is a closed cylinder
and 
\[
  \Cont^{e}(V) = \{ \gamma \in J_{\infty}X \mid \ord_{\gamma}(V) = e \} 
 = \Cont^{\geq e}(V) \smallsetminus \Cont^{\geq e + 1}(V)
\]
is a locally closed cylinder. 

Cylinders of this form are called \define{contact loci}.  For each $m \ge e$, we let  
\[
 \Cont^{\geq e}(V)_m = \psi_m ( \Cont^{\geq e}(V) ) \quad \text{and} \quad  \Cont^{e}(V)_m = \psi_m ( \Cont^{e}(V) ),
\]
denote the loci of $m$-jets with contact order  with $V$ at least $e$ and precisely $e$, respectively.

If subvarieties $V_1,\ldots,V_s$ of $X$ are specified, along with an $s$-tuple of nonnegative integers $\mathbf{e} = (e_1,\ldots,e_s)$, we write
\begin{align*}
  \Cont^{\ge \mathbf{e}}(V_\bullet) &= \bigcap_{i = 1}^{s} \Cont^{\ge e_i} (V_i)  \\
\intertext{and}
 \Cont^{\mathbf{e}}(V_\bullet) &= \bigcap_{i = 1}^{s} \Cont^{e_i} (V_i)
\end{align*}
for the corresponding \define{multi-contact loci}.

\begin{remark}
Ein, Lazarsfeld and Musta{\c{t}}{\v{a}} \cite{ELMContact} gave a correspondence between closed, irreducible cylinders of $J_{\infty} X$ and divisorial valuations of the function field of $X$.
\end{remark}

We recall some results 
relating arc spaces and singularities \cite{MusJet, MusSingularities, EMInversion}.  
Let $X$ be a $\Q$-Gorenstein variety, and let $f\colon Y \rightarrow X$ be a resolution of singularities such that the exceptional locus $E = E_1 \cup \cdots \cup E_r$ is a simple normal crossings divisor.  The relative canonical divisor has the form $K_{Y/X} = \sum_{i = 1}^r a_i E_i$, for some 
integers $a_i$, and $X$ has \define{terminal}, \define{canonical}, or \define{log canonical} 
singularities if $a_i > 0$, $a_i \ge 0$, or $a_i \ge -1$, respectively, for all $i$. 

\begin{theorem}[{\cite[Theorem 1.3]{EMInversion}}]\label{maddog}
If $X$ is a normal, local complete intersection (l.c.i.) variety, then 
it has log canonical (canonical, terminal) singularities if and only if 
$J_m X$ is pure dimensional (irreducible, normal) for all $m \ge 0$.
\end{theorem}

\begin{remark}\label{pure-dim}
In general, the closure of $J_m\sm(X)$ (the jet scheme of the smooth locus) in $J_m X$ is an irreducible component of dimension $d(m+1)$.  Thus when $J_m X$ is pure-dimensional, its dimension is $d(m+1)$.
\end{remark}

\begin{remark}\label{Goren}
A result of Elkik \cite{ElkRationalite} and Flenner \cite{FleRational} implies that a Gorenstein variety has canonical singularities if and only if it has rational singularities. 
\end{remark} 

\begin{remark}
In fact, Musta{\c{t}}{\v{a}} proves that if $X$ is a normal, l.c.i. variety with canonical (equivalently, rational) singularities, then $J_m X$ is l.c.i., reduced, and irreducible for all $m \ge 0$. 
\end{remark}

Let $X$ be a smooth variety and let $V$ be a proper, closed subscheme.  An important invariant measuring the singularities of $V$ is  the \define{log canonical threshold} $\lct(X,V)$.  We refer the reader to \cite{MusSingularities} for details. 

\begin{theorem}[{\cite[Corollary 0.2]{MusSingularities}}]\label{lct}
If $X$ is a smooth variety and $V$ is a proper, closed subscheme, then 
\[
\lct(X,V) = \dim X - \max_m \frac{ \dim J_m V}{m + 1}.
\]
Moreover, the maximum is achieved for $m$ sufficiently divisible. 
\end{theorem}

The following theorem, which was motivated by Kontsevich's theory of motivic integration, is the main ingredient in the proofs of the above results.  We will use it in Corollary~\ref{c:birational}.

\begin{theorem}[{\cite{DLGerms}}]\label{change}
Let $f\colon Y \rightarrow X$ be a proper, birational morphism between smooth varieties $Y$ and $X$. 
If $m \ge 2e$ are non-negative integers and $\Cont^e(K_{Y/X})_m \subseteq J_m Y$ denotes the locus of $m$-jets with contact order $e$ with the relative canonical divisor, then the restriction of the induced map
$f_m\colon J_m Y \rightarrow J_m X$ to  $\Cont^e(K_{Y/X})_m$, 
\[
 f_m\colon \Cont^e(K_{Y/X})_m \rightarrow f_m(\Cont^e(K_{Y/X})_m),
\]
is a Zariski-locally trivial fibration with fiber $\A^e$.
\end{theorem}

We conclude this section with a brief remark on analytification.  Any finite-type $\C$-scheme $X$ naturally determines a complex-analytic space $X^\an$ in the sense of \cite{GRbook}; in particular, one has $(J_mX)^\an$.  On the other hand, the \define{analytic jet schemes} of a complex-analytic space $Z$ may be defined analogously as
\[
  J^\an_mZ = \Hom_{\an}( \Spec \C[t]/(t^{m+1}), Z ),
\]
where $\Spec \C[t]/(t^{m+1})$ is considered as an analytic space in the obvious way.  Naturally, $J^\an_m$ is functorial for holomorphic maps of analytic spaces.  We will use the following lemma in the proof of Proposition~\ref{p:boundary}. 

\begin{lemma}\label{l:anal}
For a scheme $X$ of finite type over $\C$, we have $J^\an_m(X^\an) = (J_mX)^\an$.
\end{lemma}

\section{Equivariant geometry of jet schemes}\label{equivgeom}

Let $G$ be a linear algebraic group acting on a smooth complex variety $X$.  Functoriality of $J_{m}$ (for $m$ in $\N\cup\{\infty\}$) implies that $J_{m}G$ is an algebraic group with an induced action on $J_{m} X$ (cf. \cite[Proposition 2.6]{IshArc}).  The main result of this section is Proposition~\ref{p:boundary}, which gives a sufficient condition for the stabilizer of a point in $J_m X$ to be contractible.

\begin{example}\label{ex:t-linear}
Let a torus $T$ act on $\A^n$ via the characters $\chi_1,\ldots,\chi_n$, i.e., $t\cdot(z_1,\ldots,z_n) = (\chi_1(t) z_1, \ldots, \chi_n(t) z_n)$.  Recall that $J_m\A^n$ is identified with $n$-tuples of truncated polynomials (i.e., elements of $\C[t]/(t^{m+1})$).  The characters also define homomorphisms $J_m T \to J_m\C^*$.  Identifying $J_m\C^*$ with truncated polynomials with nonzero constant term, $J_m T$ acts on $J_m\A^n$ by
\[
  \gamma\cdot (\xi_1,\ldots,\xi_n) = (\chi_1(\gamma)\xi_1,\ldots,\chi_n(\gamma)\xi_n),
\]
where the multiplication on the RHS is multiplication of truncated polynomials.

It is also convenient to identify $J_m\A^n$ with $n\times (m+1)$ matrices, with the entries in the $k$th column corresponding to the coefficients of $t^{k-1}$.  Under this identification, the zero section $T_0 \subseteq J_m T$ acts simply by scaling the $i$th row by $\chi_i(t)\in\C^*$.  The fixed subspace $(J_m\A^n)^{T_0}$ is identified with the rows where the corresponding character is zero; note that $(J_m\A^n)^{T_0}$ is the $m^{\textrm{th}}$ jet scheme $J_m(\A^n)^T$ of the fixed locus $(\A^n)^T$.

The same discussion holds for any (possibly disconnected) diagonalizable group $H$; for finite groups, of course, there is no difference between $J_m H$ and the zero section.
\end{example}

We refer the reader to \cite{BorLinear} and \cite{springer} for basic properties of linear algebraic groups.  In particular, we will need the following fact. 

\begin{lemma}\label{contractible}
Let $U$ be a complex unipotent group, and let $\mathfrak{u}$ denote its Lie algebra.  The exponential map $\exp\colon \mathfrak{u} \rightarrow U$ is an isomorphism of complex varieties.  In particular, $U$ is contractible.   \qedhere
\end{lemma}

\noindent
Conversely, if $G$ is not unipotent, the quotient by its unipotent radical is a nontrivial reductive group; such a group retracts onto a maximal compact subgroup, so $G$ is not contractible.  In short, a linear algebraic group is contractible if and only if it is unipotent.

Let $e$ denote the identity element of $G$ and, for any $m \ge 0$,   consider the projection $\pi_m\colon J_m G \rightarrow G$ 
and the associated exact sequence of algebraic groups
\begin{align}\label{e:semidirect}
  1 \to \pi_m^{-1}(e) \to \J_m G \xrightarrow{\pi_m} G \to 1.
\end{align}
The zero section $s_m \colon  G \rightarrow \J_m G$ (see Section~\ref{arcjet})
identifies $\J_m G$ with the semidirect product $\pi_m^{-1}(e) \rtimes G$.

The following lemma is stated in the Appendix in \cite{MusJet}.

\begin{lemma}
For any $m \ge 0$, the kernel $\pi_m^{-1}(e)$ of the projection $\pi_m\colon  J_m G \to G$ is a unipotent group. \qedhere
\end{lemma}

\noindent
The easy proof was related to us by  Musta{\c{t}}{\v{a}}; one uses induction on $m$, the exact sequence 
\[
1 \rightarrow T_e G  \rightarrow \pi_m^{-1}(e) \rightarrow \pi_{m - 1}^{-1}(e) \rightarrow 1, 
\]
and the fact that an extension of a unipotent group by another unipotent group is unipotent.

\begin{lemma}\label{l:max-torus}
Let $G$ be a linear algebraic group, with maximal torus $T$.  Then the zero section $T_0\subseteq J_m T \subseteq J_m G$ is a maximal torus of $J_m G$.
\end{lemma}

\begin{proof}
This is a general fact about unipotent extensions: Suppose $G = G'/U$, with $U\subseteq G'$ unipotent; then a torus in $G'$ is maximal if and only if its image in $G$ is maximal.  Since every torus in $G'$ intersects $U$ trivially, and hence maps isomorphically to $G$, one implication is obvious.  For the other, let $T'\subseteq G'$ be a maximal torus, let $T\subseteq G$ be a maximal torus containing the image of $T'$, and let $H'\subseteq G'$ be the preimage of $T$, so $H'$ is solvable.  Then $H'/U = T$, so a maximal torus of $H'$ has the same dimension as $T$.  It follows that $T$ is the image of $T'$.  (To obtain the statement of the lemma, put $G'=J_m G$ and $T'=T_0$.)
\end{proof}

\begin{proposition}\label{p:boundary}
Let $G$ be a connected linear algebraic group acting on a smooth variety $X$, and let $D\subseteq X$ be a $G$-invariant closed subset, with irreducible components $\{D_i\}$.  Assume that the action of $G$ on $X\setminus D$ has unipotent stabilizers.  Then $J_m G$ acts on $J_m X \setminus \bigcup J_m D_i$ with unipotent stabilizers.
\end{proposition}

\begin{proof}
We proceed by first reducing to the case where $G$ is a torus, and then to the case where $X$ is affine space.

Suppose the stabilizer $\Gamma \subseteq J_m G$ of $x_m \in J_m X$ is not unipotent, and let $\gamma_m \in \Gamma$ be a nontrivial semisimple element fixing $x_m$.  We wish to show that $x_m$ lies in $J_m D_i$, for some irreducible component $D_i \subseteq D$.

Choose a maximal torus $T\subseteq G$, so the zero section $T_0 \subseteq J_m T$ is a maximal torus in $J_m G$.  Since $\gamma_m$ is semisimple, it lies in a maximal torus of $J_m G$ (\cite[Theorem 6.4.5]{springer}).  Since all maximal tori are conjugate (\cite[Theorem 6.4.1]{springer}), there is an element $c\in J_m G$ such that $c\gamma_m c^{-1} \in T_0$.  This fixes $c\cdot x_m$, and since each irreducible component $D_i$ is $G$-invariant, $x_m$ lies in $J_m D_i$ if and only if $c\cdot x_m$ does.  Therefore we may assume $\gamma_m$ lies in the torus $T_0$.  Let $H_0 = \Gamma \cap T_0$ be the subgroup of $T_0$ fixing $x_m$; this is a diagonalizable group containing $\gamma_m$.  Write $H \subseteq T$ for its isomorphic image in $G$.

Let $x=\pi_m(x_m)\in X$.  By assumption, $\pi_m(\gamma_m)$ fixes $x$, so $x$ lies in $D$.  Let $K\subseteq H$ be the maximal compact subgroup.  Since $H$ is reductive, we have an equality of fixed point sets $X^H = X^K$.  Using the slice theorem (see \cite[I.2.1]{Audin} or \cite[Corollary 1.5]{KRlinear}) together with Lemma~\ref{l:anal}, we may replace $X$ with a $K$-invariant analytic neighborhood of $x$, and assume $X=\A^n$ with $H$ acting linearly by characters $\chi_1,\ldots,\chi_n$.  Since the fixed locus $(\A^n)^H$ is irreducible and contained in $D$, we have $(\A^n)^H \subseteq D_i$, for some $i$. Using Example \ref{ex:t-linear}, we conclude that $x_m \in (J_m\A^n)^{H_0} = J_m(\A^n)^H \subseteq J_m D_i$. 
\end{proof}

\begin{remark}
The use of the (non-algebraic) compact subgroup in the last paragraph of the proof may be slightly unsatisfying to some tastes.  However, a naive application of the natural algebraic replacement---the \'etale slice theorem---does not work, since \'etale maps do not preserve irreducibility.
\end{remark}

\section{Jet schemes and equivariant cohomology}\label{jet&eq}

In this section, we relate the equivariant cohomology ring $H^*_G X$
of a connected linear algebraic group $G$ acting on a smooth complex variety $X$ of dimension $d$, with 
the geometry of the jet schemes $J_m X$ of $X$, and 
prove a criterion for producing a geometric $\Z$-basis for $H^*_G X$.

We will use the following lemma freely throughout the rest of the paper; its proof is immediate from Lemma~\ref{l:homotopy} and the fact that when $X$ is smooth,
the morphisms $\pi_{m}\colon J_m X \rightarrow X$ and $\pi_m\colon  J_m G \rightarrow G$ are fiber bundles with contractible fibers (Lemma~\ref{smootharc}).  When $m = \infty$, we may and will define 
$H_{J_\infty G}^* J_\infty X$ to be $H_G^* J_\infty X$. 

\begin{lemma}\label{l:cohomology}
Let $X$ be a smooth $G$-variety.  For any $m \in \N \cup \{\infty\}$, we have isomorphisms
\begin{equation*}
H^*_G X  \xrightarrow{\sim} H_{G}^* \J_m X  \xrightarrow{\sim} H_{\J_m G}^* \J_m X. \qedhere
\end{equation*}
\end{lemma}

For a $G$-invariant (or $J_m G$-invariant) closed subvariety $Z \subseteq J_m X$, we let $[Z]$ denote the corresponding class in $H_G^*X$ under the isomorphism of Lemma~\ref{l:cohomology}.  
Observe that a closed cylinder $C = \psi_m^{-1} (S)$, for some $S \subseteq J_m X$, is $G$-invariant (or $J_{\infty} G$-invariant) if and only if $S$ is $G$-invariant (respectively, $J_m G$-invariant).  In this case, it follows from Lemma~\ref{l:cohomology} that there is a well-defined class $[C] = [S] \in H^*_G X$.

The following lemma is a direct application of  Lemma~\ref{l:basis} and Lemma~\ref{l:openBM}. 

\begin{lemma}\label{l:basis2}
Let $G$ be a connected linear algebraic group acting on a smooth complex variety $X$, with $D\subseteq X$ a $G$-invariant closed subset with irreducible components $D_1, \ldots, D_t$. 
Suppose there exists a filtration by $J_m G$-invariant closed subvarieties 
\[
 Z_{s} \subseteq \cdots \subseteq Z_{0} = J_m X \setminus \bigcup_i J_m D_i, 
\]
such that each $U_j = Z_j \setminus Z_{j + 1}$ has trivial equivariant Borel-Moore homology (see Definition~\ref{def:trivialBM}).  Setting $k = \min\{ \codim(Z_{s}, J_m X),\, \min\{\codim(J_m D_i, J_m X)\} \} - 1$, we have
\[
  H_G^{\leq 2k} X = \bigoplus_{\codim U_j \leq k} \Z\cdot [\overline{U}_{j}].
\]
\end{lemma}

\begin{remark}\label{rate}
Lemma~\ref{thin} implies that $\lim_{m \to \infty} \codim(J_m D_i , J_m X) = \infty$.  In fact,
Theorem~\ref{lct} implies that $\codim( J_m D_i , J_m X) \ge (m + 1) \lct(X,D_i)$, and equality is achieved for $m$ sufficiently divisible. 
\end{remark}

In order to state our results, we introduce the following notation. 
Recall from Example~\ref{ex:trivialBM} that a $G$-variety $S$ is an \define{affine family of $G$-orbits} if there is a smooth map $S \rightarrow \A^n$, and $S$ is identified with a geometric quotient of $G\times\A^n$ by some closed subgroup scheme over $\A^n$.

\begin{definition}\label{d:contractible}
Let $D\subseteq X$ be a $G$-invariant closed subset with irreducible components $D_1, \ldots, D_t$.  
A locally closed cylinder $C \subseteq J_{\infty} X$ is an \define{affine family of orbits} (with respect to $D$) if $C = \psi_m^{-1}(S)$ for some $S \subseteq J_m X$, such that $S \cap J_m D_i = \emptyset$ for all $i$, and $S$ is an affine family of $J_{m} G$-orbits.
\end{definition}

\begin{remark}\label{r:unistab}
With the notation above, suppose that $G$ acts on $X \setminus D$ with unipotent stabilizers. 
By Proposition~\ref{p:boundary} and Lemma~\ref{contractible}, the stabilizer of $x \in S \subseteq  J_m X \setminus \bigcup_i J_m D_i$ is contractible, so Example~\ref{ex:trivialBM} shows that $S$
has trivial equivariant Borel-Moore homology.  Moreover, for any $m' \ge m$, $\pi_{m',m}^{-1}(S) \subseteq J_{m'} X \setminus \bigcup_i J_{m'} D_i$ is smooth and hence has trivial equivariant Borel-Moore homology by Lemma~\ref{smootharc} and Lemma~\ref{l:homotopy}.
\end{remark}

\begin{definition}\label{d:paving}
With the notation of Lemma~\ref{l:basis2}, a decomposition  $J_\infty X \setminus J_\infty D = \bigcup_j U_j$ into a non-empty, disjoint union of cylinders is an \define{equivariant affine paving} if there exists a filtration  
\[
 J_\infty D \subseteq \cdots \subseteq Z_{j + 1} \subseteq Z_{j} \subseteq \cdots \subseteq Z_{0} = J_\infty X 
\]
by $J_\infty G$-invariant closed cylinders in $J_\infty X$ containing $J_\infty D$ such that $U_j = Z_j \setminus Z_{j+1}$ is an affine family of orbits. 
\end{definition}

We are now ready to present our first main theorem.

\begin{theorem}\label{t:basis}
Let $G$ be a connected linear algebraic group acting on a smooth complex variety $X$, with $D\subseteq X$ a $G$-invariant closed subset such that $G$ acts on $X\setminus D$ with unipotent stabilizers.  If $J_\infty X \setminus J_\infty D = \bigcup_j U_j$ is an equivariant  
affine paving, then 
\[
  H_G^*X = \bigoplus_j \Z\cdot [\overline{U}_j].
\]
\end{theorem}

\begin{proof}
We assume the notation of Definition~\ref{d:paving}.  Fix a degree $k$, and note that the filtration is either finite or satisfies $\lim_{j\to\infty} \codim U_j = \infty$ by Lemma~\ref{l:codimension}; therefore the set $\{j \,|\, \codim U_j \leq k\}$ is always finite.  Let $s-1$ be the largest index in this finite set (so $\codim Z_{s} > k$ by Lemma~\ref{l:codimension}).  
Now choose $m$ large enough so that $Z_j = \psi_m^{-1}(\psi_m (Z_j))$ for $j \leq s$, and $U_j = \psi_m^{-1}(S_j)$ for $j<s$, where $S_j\subseteq J_m X \setminus \bigcup J_m D_i$ is an affine family of $J_m G$-orbits.  Also choose $m$ large enough so that $2\min\{\codim(J_m D_i, J_m X)\} >k$ (see Remark~\ref{rate}), where the $D_i$ are the irreducible components of $D$.  Setting $Z_j' = \psi_m(Z_j) \setminus  \bigcup_i J_m D_i$, we have a filtration of $J_m G$-invariant closed subvarieties
\[
 Z_{s}' \subseteq \cdots \subseteq Z_{0}' = J_m X \setminus  \bigcup_i J_m D_i, 
\]
such that each $\psi_m(U_j) = Z_j' \setminus Z_{j + 1}'$ has trivial equivariant Borel-Moore homology
by Remark~\ref{r:unistab}.  The result now follows from Lemma~\ref{l:basis2}. 
\end{proof}

\begin{remark}
If $G$ acts on a smooth variety $X$ with a free, dense open orbit $U$, then $G$ acts on $U$ with trivial, and hence unipotent, stabilizers.  Applications of Theorem~\ref{t:basis} of this type are given in Section~\ref{s:toric} and Section~\ref{s:gln}.  
\end{remark}

\begin{remark}
The simplest type of cylinder which is an affine family of orbits consists of a single $J_\infty G$-orbit in $J_\infty X$.  The existence of an equivariant affine paving involving only cells of this type is quite restrictive, however.  Indeed, suppose $X$ is compact and $G$ acts freely on $X\setminus D$.  The valuative criterion for properness \cite[Theorem II.4.7]{HarAlgebraic} implies that there is a bijection between $J_{\infty}G$-orbits of $J_{\infty} X \setminus  \bigcup_i J_\infty D_i$ and elements of the \emph{affine Grassmannian} $G((t))/G[[t]]$, and the latter is uncountable unless $G$ is diagonalizable.  Since our notion of paving assumes countably many orbits---in fact, finitely many in any given codimension---essentially the only examples of this type are compactifications of tori, i.e. toric varieties (see Section~\ref{s:toric}).
\end{remark}

For the remainder of the section, we will consider a proper, equivariant birational map $f\colon Y \rightarrow X$ between smooth $G$-varieties $Y$ and $X$, for some connected linear algebraic group $G$.  We will apply our results above to describe a method for comparing the $G$-equivariant cohomology of $X$ and $Y$. 

Suppose that $D\subseteq X$ is a $G$-invariant closed subset such that $G$ acts on $X\setminus D$ with unipotent stabilizers, and, with the notation of Definition~\ref{d:paving}, consider an equivariant affine paving $J_\infty X \setminus J_\infty D = \bigcup_j U_j$.  Recall that the relative canonical divisor $K_{Y/X}$ on $Y$ is the divisor defined by the vanishing of the 
Jacobian of $f\colon Y \rightarrow X$, and that $f_\infty\colon J_\infty Y \rightarrow J_\infty X$ denotes 
the morphism of arc spaces corresponding to $f$.  We say that the paving is \define{compatible} 
with $f$ if $f_{\infty}^{-1} (U_j) \subseteq \Cont^{e_j} (K_{Y/X})$ for some non-negative integer 
$e_j$ and for all $j$.  In this case, we will write $e_j = \ord_{f_{\infty}^{-1} (U_j) } (K_{Y/X})$. 

\begin{corollary}\label{c:birational}
Let $G$ be a connected linear algebraic group and let  $f\colon Y \rightarrow X$ be a proper, equivariant birational map between smooth $G$-varieties $Y$ and $X$.  Let $D\subseteq X$ be a $G$-invariant closed subset such that $G$ acts on $X\setminus D$ with unipotent stabilizers, and let $J_\infty X \setminus J_\infty D = \bigcup_j U_j$ be an equivariant affine paving which is compatible with $f$.  These data determine a bijection between $\Z$-bases of $H_G^* Y$ and $H_G^*X$, explictly given by
\[
  H_G^*Y = \bigoplus_j \Z\cdot [\overline{f_{\infty}^{-1}(U_j)}], \; \;  H_G^*X = \bigoplus_j \Z\cdot [\overline{U_j}].
\]
Moreover, $\codim U_j = \codim f_{\infty}^{-1}(U_j) + e_j$, where $e_j= \ord_{f_{\infty}^{-1} (U_j) } (K_{Y/X})$. 
\end{corollary}

\begin{proof}
The paving of Definition~\ref{d:paving},
\[
 J_\infty D \subseteq \cdots \subseteq Z_{j + 1} \subseteq Z_{j} \subseteq \cdots \subseteq Z_{0} = J_\infty X ,
\]
lifts to a chain of $J_\infty G$-invariant closed cylinders containing $J_\infty (f^{-1} (D))$:
\[
 f_\infty^{-1} (J_\infty D) = J_\infty (f^{-1} (D)) \subseteq \cdots \subseteq f_\infty^{-1}(Z_{j + 1}) \subseteq f_\infty^{-1}(Z_{j}) \subseteq \cdots \subseteq f_\infty^{-1}(Z_{0}) = J_\infty Y.
\] 
Here $f_\infty^{-1}(Z_{j}) \setminus  f_\infty^{-1}(Z_{j + 1}) = f_\infty^{-1}(U_{j})$, and $f^{-1} (D)$ denotes the scheme-theoretic inverse image of $D$.  

Fix a degree $k$ and  note that the set $\{j \,|\, \codim f_\infty^{-1}(U_j) \leq k\}$ is finite by Lemma~\ref{l:codimension}.  
Let $s - 1$ be an index greater than $\max(\{j \,|\, \codim f_\infty^{-1}(U_j) \leq k\})$ and $\max(\{j \,|\, \codim U_j \leq k\})$.  By the proof of Theorem~\ref{t:basis} and Lemma~\ref{l:basis2}, we may choose $m$ sufficiently large such that 
  we have a filtration of $J_m G$-invariant closed subvarieties
\[
Z_{s}' \subseteq \cdots \subseteq Z_{0}' = J_m X \setminus  \bigcup_i J_m D_i, 
\]
such that each $\psi_m^X(U_j) = Z_j' \setminus Z_{j + 1}'$ has trivial Borel-Moore homology and $U_j = \psi_m^{-1}(\psi_m^X(U_j))$.  Moreover, if $D_1, \ldots, D_t$ denote the irreducible components of $D$, then we may choose $m$ large enough so that $2\min\{\codim(J_m f^{-1}(D_i), J_m Y)\} >k$ (by Lemma~\ref{thin}) and $m \ge 2e_j$ for $0 \le j \le s - 1$.

Consider the filtration 
\[
f_{m}^{-1} (Z_{s}') \subseteq \cdots \subseteq f_{m}^{-1}(Z_{0}') =  J_m Y \setminus \bigcup_i J_m f^{-1}(D_i), 
\]
with $f_{m}^{-1} (Z_j') \setminus f_{m}^{-1} (Z_{j + 1}') = f_{m}^{-1} (\psi_m^X(U_j)) = \psi_m^Y( f_{\infty}^{-1} (U_j))$.  By Theorem~\ref{change}, the restriction 
$f_m\colon  f_m^{-1}(U_j) \rightarrow U_j$ is a $J_m G$-equivariant, Zariski-locally trivial 
fibration with fiber $\A^{e_j}$.  We conclude that $f_m^{-1}(U_j)$ has trivial equivariant 
Borel-Moore homology and $\codim U_j = \codim f_{\infty}^{-1}(U_j) + e_j$.  Using Lemma~\ref{l:basis} and Lemma~\ref{l:openBM}, we conclude that $H_G^{2k - 1} Y = 0$ and 
\[
  H_G^{2k} Y = \bigoplus_{\codim U_j + e_j = k } \Z\cdot [\overline{f_\infty^{-1}(U_j)}].
\]
The result now follows from Theorem~\ref{t:basis}. 
\end{proof}

In the succeeding two sections, we will give criteria to interpret multiplication in the equivariant cohomology ring geometrically.  An answer to the following question may be very useful in proving Conjecture~\ref{c:gln} (cf. Example~\ref{e:fullflag} and Remark~\ref{r:degrees}):

\begin{question}
Under suitable assumptions, can one compare the multiplication of classes in the $\Z$-bases of $H^*_G Y$ and $H^*_G X$ determined by Corollary~\ref{c:birational}?
\end{question} 

\begin{remark}
The relationship between the graded dimensions of $H_G^*(Y;\C)$ and  $H_G^*(X;\C)$, and the relative canonical divisor $K_{Y/X}$, would be predicted by an equivariant version of motivic integration. 
We say that two smooth $G$-varieties $X$ and $Y$ are \emph{equivariantly $K$-equivalent} if there is a smooth $G$-variety $Z$ and $G$-equivariant, proper birational maps $Z \rightarrow X$ and $Z \rightarrow Y$ such that $K_{Z/X} = K_{Z/Y}$. For example, one may consider equivariantly $K$-equivalent toric varieties with respect to the torus action  (cf. Section~\ref{s:toric}).  As in the non-equivariant case, one expects that if 
$X$ and $Y$ are equivariantly $K$-equivalent, then $\dim_{\C} H^i_G(X;\C) = \dim_{\C} H^i_G(Y;\C)$ for all $i \ge 0$.
\end{remark}

\begin{question}
Do there exist interesting examples of $G$-equivariantly $K$-equivalent varieties, where $G$ is non-trivial, other than $K$-equivalent toric varieties?
\end{question}

\section{Multiplication of classes I}\label{s:mult1}

In this section and the next, we use jet schemes to give a geometric interpretation of multiplication in the equivariant cohomology ring $H^*_G X$ of a smooth variety $X$ acted on by a connected linear algebraic group $G$.  We present two sets of results, with different assumptions on the singularities of subvarieties: the first concerns local complete intersection varieties (treated in this section), and the second requires the singular locus to be sufficiently small (discussed in the following section).

It will be convenient to introduce some terminology for this section.  A subvariety $V\subseteq X$ is an \define{equivariant complete intersection} if it has codimension $r$ and is the scheme-theoretic intersection of $r$ $G$-invariant hypersurfaces in $X$.  Similarly, $V \subseteq X$ is an \define{equivariant local complete intersection (e.l.c.i.)} if it is a local complete intersection variety locally cut out by $G$-invariant hypersurfaces.  Of course, a $G$-invariant l.c.i.~subvariety need not be e.l.c.i.: for example, the origin in $\C^n$ is not cut out by $GL_n$-invariant hypersurfaces (since there are no such hypersurfaces).

For a tuple of non-negative integers $\mm = (m_1,\ldots,m_s)$, let $\lambda(\mm)=(\lambda_1,\ldots,\lambda_s)$ be the partition defined by $\lambda_i = m_i+\cdots+m_s$.  The main theorem of this section is this:

\begin{theorem}\label{t:multiplication}
Assume the following:
\renewcommand{\theenumi}{\fnsymbol{enumi}}
\begin{enumerate}
\item $G$ is a connected reductive group, $X^G$ is finite, and the natural map $\iota^*\colon H_G^*X \to H_G^*X^G$ is injective. \label{hypothesis}
\end{enumerate}
Consider a chain of e.l.c.i. subvarieties
\[
  V_s \subseteq V_{s-1} \subseteq \cdots \subseteq V_1 \subseteq X,
\]
and a tuple $\mm = (m_1,\ldots,m_s)$ of non-negative integers.  If $\codim \Cont^{\ge \lambda(\mm)}(V_\bullet)  = \sum_{i = 1}^s m_i \codim V_i$, then 
\begin{equation}\label{e:mult}
[V_1]^{m_1}\cdots [V_s]^{m_s} = [\Cont^{\ge \lambda(\mm)}(V_\bullet)]
\end{equation}
\end{theorem}

\begin{remark}\label{r:codimension-closure}
In the statement of the above theorem, observe that if the hypothesis $\codim \Cont^{\ge \lambda(\mm)}(V_\bullet)  = \sum_{i = 1}^s m_i \codim V_i$ holds for all tuples $\mm = (m_1,\ldots,m_s)$ of non-negative integers, then  $[\Cont^{\ge \lambda(\mm)}(V_\bullet)] = [\overline{\Cont^{\lambda(\mm)}(V_\bullet)}]$. 
\end{remark}

We will prove Theorem~\ref{t:multiplication} by reducing to the case of $\A^d$.  The assumption \eqref{hypothesis} is needed only for the reduction, so we do not require it in what follows, when $X=\A^d$.

Let $G$ act on $\A^d$ and let $V \subseteq \A^d$ be a $G$-invariant hypersurface, defined 
by $f \in  \C[x_1,\ldots,x_n]$.  Recall from  Example~\ref{e:equations} that 
$J_m V \subseteq J_m \A^d$ is defined by equations $\{ f^{(k)} \mid 0 \le k \le m \}$ 
in the variables $\{ x_i^{(k)} \mid 1 \le i  \le n,\, 0 \le k \le m \}$. 

\begin{lemma}\label{l:hypersurface}
For $0 \le k \le m$, the hypersurface $V^{(k)} := \{ f + f^{(k)}  = 0 \} \subseteq J_m \A^d$  is $G$-invariant, and under the isomorphism of Lemma~\ref{l:cohomology}, 
$[V^{(k)}] = [V] \in H_G^* \A^d$. 
\end{lemma}
\begin{proof}
The lemma is trivial when $k = 0$, so assume $k \ge 1$. 
Since $V$ is invariant, $g \cdot f = \lambda(g) f$ for some character $\lambda \colon  G \to \C^*$.  With the notation of Example~\ref{e:equations}, it follows from the definition of the action of $G$ on $J_m \A^d$ that 
\[
  (g \cdot f)\left(\sum_{k = 0}^m x_1^{(k)} t^k, \ldots, \sum_{k = 0}^m x_d^{(k)} t^k\right) 
  = \lambda(g) f\left(\sum_{k = 0}^m x_1^{(k)} t^k, \ldots, \sum_{k = 0}^m x_d^{(k)} t^k\right). 
\]
In particular, considering coefficients of $t^k$ on both sides gives 
$g \cdot f^{(k)} = \lambda(g) f^{(k)}$ for $0 \le k \le m$, and we conclude that $V^{(k)}$ is $G$-invariant. 

Let $\mathcal{V} \subseteq J_m\A^d \times \A^1$ be defined by the equation $f + \zeta f^{(k)} = 0$ (where $\zeta$ is the parameter on $\A^1$).  Thus $\mathcal{V} \to \A^1$ is an equivariant family of hypersurfaces in $J_m\A^d$, whose fibers at $\zeta=0$ and $\zeta=1$ are $V$ and $V^{(k)}$, respectively.  (The polynomials $f$ and $f^{(k)}$ involve different variables, so $f+\zeta f^{(k)}$ is never identically zero; hence each fiber has the same dimension.)  Since $\mathcal{V}$ is a hypersurface in an affine space, it follows that the projection $\mathcal{V} \to \A^1$ is flat; indeed, one easily checks that $\C[\mathcal{V}]$ is torsion free and hence free over $\C[\zeta]$.  We conclude that $[V^{(k)}]=[V]$.
\end{proof}

\begin{remark}\label{r:multi}
If $G \cong (\C^*)^r$ is a torus, then the equivariant cohomology class of a torus-invariant subvariety $V \subseteq \A^d$ is equal to its \emph{multi-degree} \cite[Chapter 8]{MSCombinatorial}. In this case, it follows from the description of $f^{(k)}$ as an iterated derivation of $f$ in Example~\ref{e:equations} that $V$ and $V^{(k)}$ have the same multi-degree, implying the above lemma.
\end{remark}

Recall that for a tuple of non-negative integers $\mm = (m_1,\ldots,m_s)$, we let $\lambda(\mm)=(\lambda_1,\ldots,\lambda_s)$ be the partition defined by $\lambda_i = m_i+\cdots+m_s$.  

\begin{proposition}\label{p:affine-nested}
Consider a chain of equivariant complete intersection subvarieties
\[
  V_s \subseteq V_{s-1} \subseteq \cdots \subseteq V_1 \subseteq \A^d,
\]
and a tuple $\mm = (m_1,\ldots,m_s)$ of non-negative integers.  If $\codim \Cont^{\ge \lambda(\mm)}(V_\bullet)  = \sum_{i = 1}^s m_i \codim V_i$, then 
\[
[V_1]^{m_1}\cdots [V_s]^{m_s} = [\Cont^{\ge \lambda(\mm)}(V_\bullet)] .
\]
\end{proposition}

\begin{proof}
We will show that $\Cont^{\ge \lambda(\mm)}(V_\bullet)$ is an equivariant complete intersection.  Fix $m \ge \lambda_1 - 1$, so the equations defining $\Cont^{\ge \lambda(\mm)}(V_\bullet)$ are the same as those defining $\bigcap\pi_{m,\lambda_i-1}^{-1}(J_{\lambda_i-1}V_i)$ in $J_m\A^d$.  It will suffice to prove the claimed equation in $H_G^*J_m\A^d$.

For each $i$, let $r_i = \codim V_j$ and let $f_{i,1},\ldots,f_{i,r_i}$ be (semi-invariant) polynomials defining $V_i$.  Thus
\[
  \{ f_{i,j}^{(k)} \mid 1\leq j\leq r_i, \, 0\leq k\leq \lambda_i-1 \}
\]
defines $J_{\lambda_i-1}V_i$ in $J_{\lambda_i-1}\A^d$, as well as $\pi_{m,\lambda_i-1}^{-1}(J_{\lambda_i-1}V_i)$.

Now consider $V_s \subseteq V_{s-1}$.  Since $J_{\lambda_s-1}V_s \subseteq J_{\lambda_s-1}V_{s-1}$, we have a containment of ideals
\[
  ( f_{s,j}^{(k)} \mid 1\leq j\leq r_s, \, 0\leq k\leq \lambda_s-1 ) \supseteq  (f_{s-1,j}^{(k)} \mid 1\leq j\leq r_{s-1}, \, 0\leq k\leq \lambda_s-1 ).
\]
To cut out $\pi_{m,\lambda_s-1}^{-1}(J_{\lambda_s-1}V_s) \cap \pi_{m,\lambda_{s-1}-1}^{-1}(J_{\lambda_{s-1}-1}V_{s-1})$, then, we need the $m_s\cdot r_s$ equations
\[
  \{f_{s,j}^{(k)} \mid 1\leq j\leq r_s, \, 0\leq k\leq \lambda_s-1\}
\]
together with the $m_{s-1}\cdot r_{s-1}$ equations
\[
  \{f_{s-1,j}^{(k)} \mid 1\leq j\leq r_{s-1}, \, \lambda_s\leq k\leq \lambda_{s-1}-1\}.
\]
Continuing in this way, we obtain $\sum_{i=1}^s m_i\cdot r_i$ equations defining $\bigcap\pi_{m,\lambda_i-1}^{-1}(J_{\lambda_i-1}V_i)$; by hypothesis, this is the codimension of $\bigcap\pi_{m,\lambda_i-1}^{-1}(J_{\lambda_i-1}V_i)$, so it is a complete intersection.  It follows that
\[
  [ \bigcap\pi_{m,\lambda_i-1}^{-1}(J_{\lambda_i-1}V_i) ] = \prod_{i=1}^s \prod_{j=1}^{r_i} \prod_{k=\lambda_{i-1}}^{\lambda_i-1} [V_{i,j}^{(k)}],
\]
where $V_{i,j}^{(k)}\subseteq J_m\A^d$ is the $G$-invariant hypersurface defined by $f_{i,j}^{(k)}$.  By Lemma \ref{l:hypersurface}, the class $[V_{i,j}^{(k)}]$ is independent of $k$, and since $V_i$ is a complete intersection, we have $\prod_{j=1}^{r_i} [V_{i,j}^{(0)}] = [V_i]$.  The proposition follows.
\end{proof}

In practice, the codimension condition in the above proposition may be difficult to check.  It would be very interesting to have a nice answer to the following question. 

\begin{question}
Can one give a geometric criterion  for the codimension condition in Proposition~\ref{p:affine-nested} to be satisfied for all tuples $\mm = (m_1,\ldots,m_s)$ of non-negative integers?
\end{question}
  
  In the case when $V_s = V_1 = V \subseteq \A^d$, we have the following answer. 

\begin{corollary}\label{c:affine-CI}
Suppose $V\subseteq \A^d$ is an equivariant complete intersection.  Then $[J_m V] = [V]^{m+1}$ whenever $J_m V$ is pure-dimensional.  In particular, if $V$ is normal and $[V]$ is not nilpotent in $H^*_G X$, then this equation holds for all $m\geq 0$ if and only if $V$ has log canonical singularities.
\end{corollary}

\begin{proof}
The first statement follows from Proposition~\ref{p:affine-nested}, and the second is immediate from Theorem~\ref{maddog}.
\end{proof}

\begin{proof}[Proof of Theorem \ref{t:multiplication}]
By hypothesis \eqref{hypothesis}, $H_G^*X$ embeds in $H_G^*X^G$, so it suffices to establish the formula \eqref{e:mult} after restriction to a fixed point $p\in X^G$.  Since $G$ is reductive, the  slice theorem gives a $G$-invariant (\'etale or analytic) neighborhood of $p$ equivariantly isomorphic to $\A^d$.  Now apply Proposition~\ref{p:affine-nested}.
\end{proof}

\begin{remark}\label{r:q-coeff}
If one uses $\Q$ coefficients for cohomology, the hypothesis \eqref{hypothesis} in Theorem~\ref{t:multiplication} can be replaced by the following:
\renewcommand{\theenumi}{$*'$}
\begin{enumerate}
\item $G$ is connected, and for a maximal torus $T\subseteq G$, $X^T$ is finite and the map $H_T^*X \to H_T^*X^T$ is injective.\label{hypothesis2}
\end{enumerate}
Moreover, we may assume that our subvarieties $\{ V_i \}$ are e.l.c.i with respect to $T$. 
Indeed, \eqref{hypothesis} applies to $T$, and $H_G^*X$ embeds in $H_T^*X$ as the subring of Weyl invariants, by a theorem of Borel. 

Corollary \ref{c:affine-CI} also extends to e.l.c.i. subvarieties, using either hypothesis \eqref{hypothesis} or \eqref{hypothesis2}.
\end{remark}

The following variant is useful; it follows immediately from Theorem~\ref{t:multiplication}.  

\begin{corollary}\label{c:multiplication}
Assume hypothesis \eqref{hypothesis}, and let $Y_1, \ldots, Y_s$ be invariant subvarieties of $X$ such that each intersection $V_i = Y_1 \cap \cdots \cap Y_i$ is proper and e.l.c.i.  For a tuple $\mm = (m_1,\ldots,m_s)$ of non-negative integers, if $\codim \Cont^{\ge \mm}(Y_\bullet) = \sum_{i = 1}^s m_i \codim Y_i$, then
\[
[Y_1]^{m_1}\cdots [Y_s]^{m_s} = [\Cont^{\ge \mm}(Y_\bullet)]. 
\]
\end{corollary}

\begin{example}\label{e:ncd}
Suppose $G$ and $X$ satisfy \eqref{hypothesis}, and let $D = D_1 + \cdots + D_s$ be a $G$-invariant normal crossings divisor in $X$.  Corollary~\ref{c:multiplication} and Remark~\ref{r:codimension-closure} apply, so
\[
[D_1]^{m_1}\cdots [D_s]^{m_s} = [\Cont^{\ge \mm}(D)] = [\overline{\Cont^{\mm}(D)}].
\]
\end{example}

\section{Multiplication of classes II}\label{s:mult2}

Replacing the assumption that $V\subseteq X$ be an equivariant local complete intersection with a restriction on the dimension of the singular locus of $V$, we can prove versions of the results of the previous section.  Throughout this section, $G$ is assumed to be reductive.


In what follows, we will embed $X$ as a smooth subvariety of $J_m X$ via the zero section, and write 
$\Delta_{m+1}\colon X\hookrightarrow X \times \cdots \times X$ for the diagonal embedding of $X$ in the $(m+1)$-fold product.

\begin{lemma}\label{l:normal}
There are canonical isomorphisms 
\[  N_{X/J_m X} \cong N_{\Delta_{m + 1}/X \times \cdots \times X} \cong TX^{\oplus m}. \] 
\end{lemma}

\begin{proof}
By the functorial definition of jet schemes, 
\[
T (J_m X) = \Hom(\Spec \C[s]/(s^2), J_m X) =  \Hom(\Spec \C[s, t]/(s^2, t^{m + 1}), X),
\] 
so for a closed point $x$ in $X$, a vector in $T_x J_m X$ corresponds to a $\C$-algebra homomorphism
\[
\theta\colon  \mathcal{O}_{X,x} \rightarrow \C[s, t]/(s^2, t^{m + 1}), \; \theta(y) = \theta_0(y) + \sum_{i = 0}^{m} \phi_i(y) st^i,
\]
where $\theta_0\colon   \mathcal{O}_{X,x} \rightarrow \C$ is the $\C$-algebra homomorphism corresponding to $x$. That $\theta$ is a $\C$-algebra homomorphism is equivalent to requiring that $\theta_0(y) + s\phi_i(y)$ is a closed point in $T_x X$ for $0 \le i \le m$.
We therefore have a natural isomorphism
\begin{equation*}
  T_x J_m X \cong T_x X \times \cdots \times T_x X .
\end{equation*}
Moreover, identifying $X$ with the zero section, we have an embedding of $T_x X$ in $T_x J_m X$ whose image corresponds to the subspace where $\phi_1 = \cdots = \phi_m = 0$. 
On the other hand, 
\[
T_x (X \times \cdots \times X) \cong T_x X \times \cdots \times T_x X,
\]
and $T_x \Delta_{m + 1}$ is the image of $T_x X$ under the diagonal embedding in 
$T_x X \times \cdots \times T_x X$. 
Hence (with a slight abuse of notation)
\[
N_{X/J_m X, x} = \{ (\phi_0, \phi_1, \ldots, \phi_m) \mid \phi_i \in T_x X \} /  \{ (\phi, 0, \ldots, 0) \mid \phi \in T_x X \},
\]
\[
N_{\Delta_{m + 1}/X \times \cdots \times X, x} = \{ (\phi_0', \phi_1', \ldots, \phi_m') \mid \phi_i' \in T_x X \} /  \{ (\phi', \phi', \ldots, \phi') \mid \phi' \in T_x X \},
\]
and there is a natural isomorphism sending 
\[
 (\phi_0, \phi_1, \ldots, \phi_m) \mapsto (\phi_0, \phi_0 - \phi_1, \ldots, \phi_0 - \phi_m). 
\]
One easily verifies that this extends to a canonical global isomorphism. 
\end{proof}

For the remainder of the section we consider a chain of invariant irreducible 
subvarieties
\[
  V_s \subseteq V_{s-1} \subseteq \cdots \subseteq V_1 \subseteq X,
\]
and a tuple $\mm = (m_1,\ldots,m_s)$ of non-negative integers with $m_s > 0$.  Recall that 
$\lambda(\mm)=(\lambda_1,\ldots,\lambda_s)$ denotes the partition defined by 
$\lambda_i = m_i+\cdots+m_s$, and $\Cont^{\ge \lambda(\mm)}(V_\bullet)$ denotes the associated 
multi-contact locus.  If $U = X \setminus \bigcup_i \sing(V_{i})$, then 
$\Cont^{\ge \lambda(\mm)}(V_\bullet)$ restricts to a smooth, irreducible cylinder in 
$J_\infty U$.  The closure of this restricted cylinder in $J_\infty X$ is an irreducible 
cylinder of codimension $\sum_i m_i \codim V_i$ which we denote by 
$\Cont^{\ge \lambda(\mm)}\sm(V_\bullet)$.

\begin{remark}\label{r:normal}
Consider a chain of invariant smooth subvarieties
\[
  V_s \subseteq V_{s-1} \subseteq \cdots \subseteq V_1 \subseteq X,
\]
and a tuple $\mm = (m_1,\ldots,m_s)$ of non-negative integers with $m_s > 0$.  Fix $m \ge \lambda_1 - 1$, so that $\Cont^{\ge \lambda(\mm)}(V_\bullet)_m = \bigcap\pi_{m,\lambda_i-1}^{-1}(J_{\lambda_i-1}V_i)$ in $J_m X$.  The proof of Lemma~\ref{l:normal} gives a canonical isomorphism between the normal bundle of $V_s$, embedded via the zero section in $\Cont^{\ge \lambda(\mm)}(V_\bullet)_m$, and the normal bundle of $V_s$, embedded via the diagonal embedding
in $\underbrace{V_s \times \cdots \times V_s}_{m_s \textrm{ times } } \times \cdots \times\underbrace{V_1 \times \cdots \times V_1}_{m_1 \textrm{ times } }  \times \underbrace{X \times \cdots \times  X}_{m + 1 - \lambda_1 \textrm{ times } }$. 
\end{remark}

\begin{theorem}\label{t:smooth}
Let $X$ be a smooth $G$-variety of dimension $d$, and consider a chain of invariant 
subvarieties
\[
  V_s \subseteq V_{s-1} \subseteq \cdots \subseteq V_1 \subseteq X,
\]
and a tuple $\mm = (m_1,\ldots,m_s)$ of non-negative integers.  We have
\begin{equation*}
[V_1]^{m_1}\cdots [V_s]^{m_s} = [\Cont^{\ge \lambda(\mm)}(V_\bullet)]
\end{equation*}
whenever $\min\{ \codim (\sing(V_{r}),X) \}  > \sum_i m_i \codim (V_i,X)$.  (By convention, $\dim\emptyset = -\infty$.)
\end{theorem}

\begin{proof}
Clearly we may assume that $m_s > 0$.  Let $c_i$ denote the codimension of $V_i$.  Let $Z = \bigcup_r \sing(V_{r})$ and let $U=X\setminus Z$, so we have an exact sequence 
\[
\cdots \rightarrow \BH^G_{2(d - \sum_i m_i c_i)} Z \rightarrow  H^{\sum_i 2 m_i c_i }_G X \rightarrow  H^{\sum_i 2m_i c_i}_G U \rightarrow \BH^G_{2(d - \sum_i m_i c_i)-1} Z
\rightarrow \cdots. 
\]
By the assumption on $\dim Z$, the left and right terms are zero, so the restriction map $H^{\sum_i 2 m_i c_i }_G X \rightarrow  H^{\sum_i 2 m_i c_i }_G U $ is an isomorphism.  Replacing $X$ with $U$ and $V_r$ with $V_r \cap U$, we reduce to the case when each $V_r$ is smooth.

Fix $m \ge \lambda_1 - 1$, so that $\Cont^{\ge \lambda(\mm)}(V_\bullet)_m = \bigcap\pi_{m,\lambda_i-1}^{-1}(J_{\lambda_i-1}V_i)$ in $J_m X$.  
Let $K\subseteq G$ be a maximal compact subgroup; since a reductive group retracts onto its maximal compact subgroups, $G$- and $K$-equivariant cohomology are naturally isomorphic, and we identify the two for the rest of this argument.  The slice theorem (see \cite[I.2.1]{Audin}) gives a $K$-invariant neighborhood $U_X \subseteq J_m X$ of $X$ which is $K$-equivariantly isomorphic to a neighborhood of the zero section in $N_{X/J_m X}$.  Note that restriction to the zero section $H^*_G J_m X \rightarrow H^*_G U_X \rightarrow H^*_G X$ is an isomorphism by Lemma~\ref{l:cohomology}.  Since $U_X$ retracts onto $X$, the map $H^*_G U_X \rightarrow H^*_G X$ is also an isomorphism, and hence the restriction $H^*_G J_m X \rightarrow H^*_G U_X$ is an isomorphism.
  
By the canonical isomorphism $N_{X/J_m X} \cong N_{\Delta_{m + 1}(X)/X \times \cdots \times X}$ 
of Lemma~\ref{l:normal}, $U_X$ is ($K$-equivariantly) isomorphic to an open neighborhood of the diagonal $\Delta_{m + 1}(X)$ in $X \times \cdots \times X$.  Moreover, Remark~\ref{r:normal} 
implies that the class of $\Cont^{\ge \lambda(\mm)}(V_\bullet)_m$ in $J_m X$ restricts to the 
class of the intersection $U_{V_\bullet, m_i}$ of $U_X$ with
\[
  V_{\bullet, m_i} :=\underbrace{V_s \times \cdots \times V_s}_{m_s \textrm{ times } } \times \cdots \times\underbrace{V_1 \times \cdots \times V_1}_{m_1 \textrm{ times } }  \times \underbrace{X \times \cdots \times  X}_{m + 1 - \lambda_1 \textrm{ times } }.
\] 

Consider the commutative diagram:
\begin{diagram}
\BH^G_{2(m + 1)d-\sum_i 2m_i c_i }( V_{\bullet, m_i}) & \rTo^\sim & \BH^G_{2(m + 1)(d-c)}(U_{V_\bullet, m_i}) \\
\dTo                 &                   &   \dTo \\
H_G^{\sum_i 2 m_i c_i }(X \times \cdots \times X) &      \rTo          &   H_G^{\sum_i 2 m_i c_i }(U_X)  &\; = \;& H_G^{\sum_i 2 m_i c_i }X.
\end{diagram}
 Here the horizontal arrows are restriction to the tubular neighborhoods; the composition in the bottom row is $\Delta^*$.  Going counter-clockwise around the diagram, we have $\Delta^*[ V_{\bullet, m_i}] = [V_1]^{m_1}\cdots [V_s]^{m_s}$.  Going clockwise, we have $ [\Cont^{\ge \lambda(\mm)}(V_\bullet)]$, completing the proof.
\end{proof}

\begin{remark}
The condition on singular loci in Theorem~\ref{t:smooth} is quite restrictive.  In particular, it implies that $\codim(\bigcup_r \sing(V_{r}), X) > \sum_i m_i \codim V_i$, and hence that  $\codim \Cont^{\ge \lambda(\mm)}(V_\bullet) = \codim \Cont^{\ge \lambda(\mm)}\sm(V_\bullet) = \sum_i m_i \codim V_i$.  This condition is slightly stronger than the codimension condition in Theorem~\ref{t:multiplication}.  
\end{remark}

In the case when $V_s = V_1 = V$, Theorem~\ref{t:smooth} reduces to the following corollary. 

\begin{corollary}\label{c:mult}
Let $X$ be a smooth $G$-variety of dimension $d$, and let $V\subseteq X$ be a $G$-invariant connected subvariety of codimension $c$.  We have
\begin{equation}\label{eq:magic2}
  [V]^{m+1} = [\overline{J_m \sm(V)}]
\end{equation}
whenever $\codim (\sing(V),X) > (m + 1)c$.  (By convention, $\codim\emptyset = \infty$.)
\end{corollary}

\begin{example}\label{ex:cusp1}
Some condition on the singular locus is necessary.  For example, let $V \subseteq \A^2$ be the cuspidal cubic defined by $x^3-y^2 = 0$; this is invariant for the action of $T=\C^*$ by $z\cdot(x,y) = (z^2 x, z^3 y)$.  Since $V$ has degree $6$ with respect to the grading corresponding to the $\C^*$-action, we have $[V] = 6t$ in $H_T^*\A^2 \isom \Z[t]$.  The tangent bundle $TV = J_1 V$ is defined by the two equations $x^3 - y^2 = 0$ and $3x^2 x_1 - 2y y_1=0$, each of which has degree $6$, so $[TV] = 36t^2 = [V]^2$.  On the other hand, $TV$ has two irreducible components, and one can check that $[\overline{T\sm(V)}]=18t^2$.
%
\end{example}

\section{Higher-order multiplicities}\label{s:multiplicities}

Let $V \subseteq X$ be a $G$-invariant subvariety of codimension $c$.  As discussed in the last two sections, the discrepancy between $[V]^{m+1}$ and $[J_mV]$ bears a relation to the singularity type of $V$.  In this section, we introduce a pair of algebraic invariants measuring this discrepancy, and describe some of their properties.

Throughout this section, we will assume that the top Chern class $c^G_d(X)$ and the fundamental class $[V]$ are nonzerodivisors in $H_G^*X$.  This hypothesis holds in the important case where $G$ is a torus acting linearly on $X=\A^d$, fixing only the origin.  Having made this assumption, let $H = (H_G^*X)[c^G_d(X)^{-1}, [V]^{-1}]$ be the ring obtained by inverting these elements.

We will also abuse notation slightly by using $x$ denote both a point in $X$ and its image in $J_mX$ under the zero section $s_m\colon X \to J_mX$.

For an arbitrary variety $V$ of codimension $k$, let $(J_m V)_{\exp} \subseteq J_mV$ denote the union of all components of $J_m V$ which have ``expected dimension'' $k(m+1)$, with their induced subscheme structure.  (If $J_m V$ has embedded components of expected dimension, they should be included.)  We also write $\mJ_mV$ for the ``main component'' $\overline{J_m\sm(V)}\subseteq J_mV$, so $\mJ_mV$ is automatically pure-dimensional (of expected dimension).  
When $V \subseteq X$ is a $G$-invariant subvariety, write $c=d-k$ for the codimension, so $[(J_mV)_{\exp}]$ is a class in $H_G^{2c(m+1)}X$.

\begin{definition}
Let $V\subseteq X$ be as above.  Define the \define{global $m^\mathrm{th}$-order  equivariant multiplicities} by
\begin{align*}
  e^{G}_m(V) &= \frac{[(J_m V)_{\exp}]}{c^G_d(X)\cdot [V]^{m} } \\
 \intertext{and}
  \te^{G}_m(V) &= \frac{[\mJ_mV]}{c^G_d(X)\cdot [V]^{m} }
\end{align*}
as elements of $H$.

For a fixed point $x\in V^G$, we also define \define{(local) $m^\mathrm{th}$-order equivariant multiplicities}, as follows.  Assuming the restrictions $c_d^G(T_x X)$ and $[V]|_x$ are nonzerodivisors in $H_G^*(\pt)$, let $H_x$ be the result of inverting these elements, and set
\begin{align*}
  e_{x,m}^{G}(V) &= \frac{[(J_m V)_{\exp}]|_x}{c^G_d(T_xX)\cdot ([V]|_x)^{m} } \\
 \intertext{and}
  \te_{x,m}^{G}(V) &= \frac{[\mJ_mV]|_x}{c^G_d(T_xX)\cdot ([V]|_x)^{m} }
\end{align*}
in $H_x$.
\end{definition}

Note that $e_{x,m}^{G}(V) = \iota_x^*e^{G}_m(V)$ and $\te_{x,m}^{G}(V) = \iota_x^*\te^{G}_m(V)$, where $\iota_x \colon  \{x\}\hookrightarrow X$ is the inclusion.  Like the equivariant multiplicities described by Brion \cite[\S4]{BriChow}, these are homogeneous elements of degree $-\dim(V)$.  In fact, there is a close connection:

\begin{proposition}
For $m=0$ and $G=T$ a torus, the local multiplicities $e_{x,0}^{T}(V) = \te_{x,0}^{T}(V)$ coincide with Brion's equivariant multiplicity $e_x(V)$, as defined in \cite[\S4]{BriChow}.
\end{proposition}

\begin{proof}
Using a deformation to the normal cone, one reduces to the case where $V=C_xV$ and $X=T_xV$.  Here the requirement that $c_d(T_xX)=c_d(T_xV)$ be a nonzerodivisor is equivalent to $x$ being nondegenerate, in the terminology of \cite{BriChow}.  The local $0$-order multiplicity is
\[
  e_{x,0}^{T}(V) = \frac{[V]|_x}{c^T_d(T_xV)} =  \frac{[C_xV]|_0}{\chi_1\cdots\chi_d},
\]
where $\chi_1,\ldots,\chi_d$ are the weights of $T$ acting on $T_xV$.  This is equal to $e_x(V)$ by \cite[Theorem~4.5]{BriChow}.
\end{proof}

The results of the previous two sections have consequences for higher-order equivariant multiplicities.  For simplicity, we assume $G=T$ is a torus in what follows; an appropriate adjustment of hypotheses yields similar statements for other groups.

\begin{corollary}\label{c:simple-mult}
Suppose either 
\begin{enumerate}
\renewcommand{\theenumi}{\alph{enumi}}

\item $\codim(\sing(V),X)>(m+1)c$; or

\smallskip

\item $V \subseteq X$ is e.l.c.i., and $J_mV$ is pure-dimensional.
\end{enumerate}
\renewcommand{\theenumi}{\arabic{enumi}}
Then
\begin{align*}
  e^{T}_m(V) &= e^{T}_0(V) = \frac{[V]}{c^T_d(X) } \\
 \intertext{and, for a fixed point $x\in V^T$,}
  e_{x,m}^{T}(V) &= e_{x,0}^{T}(V) = e_x(V) .
\end{align*}

In particular, if $V$ is smooth, or e.l.c.i. and normal with log-canonical singularities, these equations hold for all $m$, by Corollaries~\ref{c:mult} and \ref{c:affine-CI}. 
\end{corollary}


A salient feature of this corollary is that the higher-order multiplicities are seen to be independent of the embedding of $V$ in $X$.  In fact, this is true more generally.

\begin{theorem}\label{t:mult-invt}
Let $V \subseteq X$ be a $T$-invariant subvariety, and let $x\in V$ be a fixed point.
\begin{enumerate}
\item The local multiplicities $e^{T}_{x,m}(V)$ and $\te^{T}_{x,m}(V)$ are independent of the embedding in $X$. \label{mult-inv1}

\smallskip

\item If $j\colon  X \hookrightarrow X'$ is an equivariant embedding of smooth varieties, then $e^{T}_m(V)_X = j^*e^{T}_m(V)_{X'}$ and $\te^{T}_m(V)_X = j^*\te^{T}_m(V)_{X'}$, where the subscript indicates which embedding of $V$ is used to define the multiplicity. \label{mult-inv2}
\end{enumerate} 
\end{theorem}

In proving \eqref{mult-inv1}, we will need to construct canonical classes in $\BH^T_{2k(m+1)}(J_m(C_xV))$.  (The same construction produces cycles for equivariant Chow groups.)  The idea is to follow the classes $[(J_mV)_{\exp}]$ and $[\mJ_mV]$ along the deformation of $V$ to the normal cone $C_xV$.  We will give the arguments in detail for the classes $e^{T}_{x,m}(V)$; they are similar for $\te^{T}_{x,m}(V)$.

Let $M=M^\circ_xV \subseteq \Bl_{x \times \infty}(V \times \P^1)$ be the total space of the deformation to the normal cone \cite[\S5]{FulIntersection}, so the projection $M \to \P^1$ is flat, with fibers identified as $M_s = V$ for $s\neq\infty$, and $M_\infty = C_xV$.  Let $p\colon J_m(M/\P^1) \to \P^1$ be the relative jet scheme (cf.~\cite{MusSingularities}), so the fiber over $s\in \P^1$ is naturally identified with $J_m(M_s)$.  Removing the fiber $p^{-1}(\infty)$, we have a subscheme
\[
  (J_mV_{\exp}) \times \A^1 \subset J_m(M/\P^1).
\]
Let $\mathcal{Z}$ be the closure.  By construction, $\mathcal{Z} \to \P^1$ is flat, of relative dimension $k(m+1)$, and the fiber over $\infty$ is a closed subscheme $Z_{x,m}(V) \subseteq J_m(M_\infty) = J_m(C_xV)$.  The class $[Z_{x,m}(V)] \in \BH^T_{2k(m+1)}J_m(C_xV)$ clearly depends only on $V$, $x$, and $m$.  When $x$ is a smooth point of $V$, note that $Z_{x,m}(V) = C_xV= T_xV$.

\begin{lemma}\label{l:zv-equal}
The image of $[Z_{x,m}(V)]$ under the map
\[
  \BH^T_{2k(m+1)}J_m(C_xV) \to \BH^T_{2k(m+1)}J_m(T_xX) = H_T^{2c(m+1)}(\pt)
\]
is the same as the image of $[(J_mV)_{\exp}]$ under the composition
\[
  \BH^T_{2k(m+1)}(J_mV)_{\exp} \to \BH^T_{2k(m+1)}(J_mX)=H_T^{2c(m+1)}(J_mX) \to H_T^{2c(m+1)}(x).
\]
\end{lemma}

\begin{proof}
The last restriction map $H_T^{2c(m+1)}(J_mX) \to H_T^{2c(m+1)}(x)$ factors through $H_T^*J_mX \to H_T^*(T_x(J_mX))$, via the specialization to the normal cone \cite[\S5.2]{FulIntersection}.  Furthermore, we have canonical isomorphisms
\[
  H_T^{2c(m+1)}(T_x(J_mX))=\BH^T_{2k(m+1)}(T_xJ_m(X))=\BH^T_{2k(m+1)}(J_m(T_xX)),
\]
so it suffices to show the two classes have equal image in this group.  

Finally, observe that the flat families $J_m(M^\circ_xX/\P^1) \to \P^1$ (the deformation constructed above, applied to $X$) and $M^\circ_x(J_mX) \to \P^1$ (the deformation to the normal cone) are naturally isomorphic.  Denote them both by $\mathcal{X}$.  Viewing $\mathcal{X}$ as $J_m(M^\circ_xX/\P^1)$, we have a flat subfamily $\mathcal{Z} \subseteq \mathcal{X}$, with $Z_0 = (J_mV)_{\exp}$ and $Z_\infty = Z_{x,m}(V)$.  Viewing $\mathcal{X}$ as $M^\circ_x(J_mX)$, we have a flat subfamily $M^\circ_x((J_mV)_{\exp}) \subseteq \mathcal{X}$, with fiber over $0$ equal to $(J_mV)_{\exp}$ and fiber over $\infty$ equal to $C_x((J_mV)_{\exp})$.  The claim now follows from a simple fact from intersection theory, Lemma~\ref{l:flat-equal} below.
\end{proof}

\begin{lemma}\label{l:flat-equal}
Suppose $\mathcal{X} \to \P^1$ is an equivariant flat family of algebraic schemes, and $\mathcal{Z}, \mathcal{V} \subseteq \mathcal{X}$ are equivariant flat (closed) subfamilies.  If $[Z_0]=[V_0]$ in $\BH^T_*(X_0)$ (or $A^T_*(X_0)$), then $[Z_\infty]=[V_\infty]$ in $\BH^T_*(X_\infty)$ (resp., $A^T_*(X_\infty)$).
\end{lemma}

The proof of this lemma is a simple exercise.  We now prove the theorem.

\begin{proof}[Proof of Theorem~\ref{t:mult-invt}]
We start with Part \eqref{mult-inv1}.  Let $T_xV$ be the Zariski tangent space, and suppose it has dimension $k'\geq k$.  Write $N_x = T_xX/T_xV$, and $c' = \dim N_x$.  By the self-intersection formula, for any class $\alpha \in H_T^*(T_xV)$, we have $\iota^*\iota_*\alpha = c^T_{c'}(N_x)\cdot \alpha$, where $\iota\colon T_xV \hookrightarrow T_xX$ is the inclusion.

Now write $\nu$ for the class of $C_xV$ in $H_T^*(T_xV)$ and $\zeta$ for the class of $Z_{x,m}(V)$ in $H_T^*(J_m(T_xV))$.  Note that these classes are intrinsic to $V$.    Using the self-intersection formula, together with the fact that the normal space to $J_m(T_xV)$ in $J_m(T_xX)$ has top Chern class equal to $c_{c'}(N_x)^{m+1}$, we have
\begin{align*}
 [Z_{x,m}(V)] &= \zeta\cdot c^T_{c'}(N_x)^{m+1} ,\\
 [C_xV]  &= \nu \cdot c^T_{c'}(N_x).
\end{align*}

Finally, the basic construction of intersection theory identifies $[V]|_x$ with $[C_xV]$ in $H_T^*(x) = H_T^*(T_xX)$.  Using these observations and Lemma~\ref{l:zv-equal}, we have
\begin{align*}
  e^{T}_{x,m}(V) &= \frac{[(J_mV)_{\exp}]|_x}{c^T_d(T_xX)\cdot ([V]|_x)^m} \\
              &= \frac{[Z_{x,m}(V)]}{c^T_d(T_xX) \cdot [C_xV]^m} \\
              &= \frac{\zeta \cdot c_{c'}(N_x)^{m+1}}{ c^T_d(T_xX) \cdot \nu^m \cdot c_{c'}(N_x)^m } \\
              &= \frac{\zeta}{c^T_{k'}(T_xV) \cdot \nu^m}.
\end{align*}
Since the last expression is intrinsic to $V$, so is the equivariant multiplicity.

The proof of Part \eqref{mult-inv2} is much easier.  Let $N_{X/X'}$ be the normal bundle for the embedding $X \hookrightarrow X'$.  By Lemma~\ref{l:reltan}, $c^T_{top}(N_{J_mX/J_mX'}) = c^T_{top}(N_{X/X'})^{m+1}$.  Therefore
\begin{align*}
 j^*[(J_mV)_{\exp}]_{X'} &= [(J_mV)_{\exp}]_{X}\cdot c^T_{top}(N_{X/X'})^{m+1},\\
 j^*[V]_{X'}  &= [V]_X \cdot c_{top}(N_{X/X'}),
\end{align*}
and substituting these into the definition of $e^T_m(V)$ proves the claimed equality.
\end{proof}





The higher-order multiplicities are already interesting, and difficult to compute, for affine plane curves $V\subseteq \A^2$, with $T=\C^*$.

\begin{example}\label{ex:node-mult}
Let $V=\{x^2-y^2=0\}$, with $T$ acting with weights $(1,1)$ (i.e., $z\cdot(a,b)=(za,zb)$).  Using induction on $m$, it is easy to show that $J_mV$ is pure-dimensional for all $m$, so we have $[J_mV]=[V]^{m+1} = (2t)^{m+1}$ for all $m$.  
%
%
On the other hand, $[\mJ_mV]=[V]=2t$ for all $m$.  So
\[
 e^T_{x,m} = {2}/{t} \quad \text{and} \quad \te^T_{x,m} = {1}/{(2^m t)}.
\]
\end{example}

\begin{example}\label{ex:cusp-mult}
As in Example~\ref{ex:cusp1}, let $V = \{x^3-y^2=0\}$, with $T$ acting with weights $(2,3)$ (i.e., $z\cdot(a,b)=(z^2a,z^3b)$).  The jet schemes $J_mV$ are pure-dimensional for $m<5$, so $[(J_m V)_{\exp}]=[V]^{m+1} = (6t)^{m+1}$ in this range.  Using Macaulay 2, we can compute $[\mJ_mV]$ for $m\leq 5$.  Noting that $c^T_2(\A^2)=6t^2$, the data are as follows:
\renewcommand{\arraystretch}{1.4}
\begin{equation*}
\begin{array}{|r||c|c|c|c|c|c|} \hline
  m            & 0   & 1      & 2      &  3    &  4     &  5 \\ \hline
 e^T_{x,m}(V)  & 1/t & 1/t    & 1/t    & 1/t   & 1/t    &  ? \\ \hline
\te^T_{x,m}(V) & 1/t & 1/(2t) & 1/(3t) & 1/(4t)& 1/(6t) & 1/(9t) \\ \hline
\end{array}
\end{equation*}
\renewcommand{\arraystretch}{1.0}

\noindent
The jet scheme $J_5V$ is not pure-dimensional; it has an irreducible component of codimension $5$ in addition to the ``main'' component $\mJ_5V$.  To compute the class $[(J_5V)_{\exp}]$, including embedded components, one needs to find the primary decomposition for the ideal of $J_5V$, which exhausted our computing capability.
\end{example}

\begin{example}
Let $V=\{x^5-y^2=0\}$, with $T$ acting by weights $(2,5)$.  Note that $[V]=10t$ and $c^T_2(\A^2)=10t^2$.  Using Macaulay 2, we compute:
\renewcommand{\arraystretch}{1.4}
\begin{equation*}
\begin{array}{|r||c|c|c|c|c|} \hline
  m            & 0   & 1      & 2      &  3       &  4      \\ \hline
 e^T_{x,m}(V)  & 1/t & 1/t    & 1/t    & 79/(50t) & ?    \\ \hline
\te^T_{x,m}(V) & 1/t & 1/(2t) & 1/(4t) & 3/(20t)  & 1/(10t)  \\ \hline
\end{array}
\end{equation*}
\renewcommand{\arraystretch}{1.0}

\noindent
The interesting entry is $e^T_{x,3}(V)$, since $J_3V$ is not pure dimensional.  According to Macaulay 2, there are three components of codimension $4$, yielding $[(J_3V)_{\exp}]=(1500+11000+3300)t^4 = 15800t^4$.
\end{example}

\section{Example: smooth toric varieties}\label{s:toric}

The goal of this section is to apply our results to give a new interpretation of 
the equivariant cohomology ring of a smooth toric variety.  We refer the reader to 
\cite{FulIntroduction} for an introduction to toric varieties.

Let $X = X(\Sigma)$ be a smooth $d$-dimensional toric variety corresponding to a fan $\Sigma$ in a lattice $N$ of rank $d$, and let $T$ be the dense torus acting on $X$.  Let $v_{1}, \ldots, v_{r}$ denote the primitive integer vectors of the rays in $\Sigma$ and let $D_{1}, \ldots, D_{r}$ denote the corresponding torus-invariant prime divisors of $X$.  
The \define{Stanley-Reisner ring} $\SR(\Sigma)$ is the quotient of $\Z[x_1, \ldots, x_r]$ by the ideal generated by monomials of the form $x_{i_1}\cdots x_{i_s}$, such that $v_{i_{1}}, \ldots, v_{i_s}$ do not span a cone in $\Sigma$. 
The equivariant cohomology ring of $X$ may be described as follows:

\begin{theorem}\cite[Theorem 8]{BDPCohomology}\label{BDP}
With the notation above, there is an isomorphism $H^*_T X \cong \SR(\Sigma)$, sending 
$[D_i]$ to $x_i$.   
\end{theorem}

Our goal is to apply our results to give a new geometric proof of this fact. 
Observe that $\SR(\Sigma)$ has a $\Z$-basis indexed by lattice points in $N$ which lie in the support $|\Sigma|$ of $\Sigma$: if $\sigma$ is a maximal cone with primitive integer vectors $v_{i_{1}}, \ldots, v_{i_d}$, then a lattice point $v = \sum_{j = 1}^d a_j v_{i_j}$ corresponds 
to the monomial $x^v := x_{i_{1}}^{a_1}\cdots x_{i_{d}}^{a_d}$ in $\SR(\Sigma)$.  In fact, 
$\SR(\Sigma)$ is isomorphic to the \define{deformed group ring} $\Z[N]^{\Sigma}$: 
this is the $\Z$-algebra with $\Z$-basis $\{ y^v \mid v \in |\Sigma| \cap N \}$ 
and multiplication defined by 
\begin{equation}\label{deformed}
y^u \cdot y^v = \left\{ \begin{array}{ll}
y^{u + v} & \textrm{ if } u,v \in \sigma \textrm{ for some } \sigma \in \Sigma, \\
0 & \textrm{ otherwise}.
\end{array} \right.
\end{equation} 
On the other hand, with the notation above, for each $v = \sum_{j = 1}^d a_j v_{i_j} \in |\Sigma| \cap N$ consider the cylinder $\Cont^{v}(D) := \bigcap_{1 \leq j \leq d} \Cont^{a_{j}} (D_{i_j})$ 
in $J_{\infty} X$.  One verifies the decomposition 
\[
  J_{\infty} X \setminus \bigcup_i J_{\infty} D_i = \coprod_{v \in |\Sigma| \cap N} \Cont^v (D). 
\]
We will let $\Cont^{\ge v}(D)$ denote the closure of $\Cont^{v}(D)$ in $J_{\infty} X$, and 
define a partial order $\leq^{\Sigma}$ on $|\Sigma| \cap N$ by setting $v \leq^{\Sigma} w$ if 
$w - v$ lies in some maximal cone in $\Sigma$ containing $v$ and $w$.  The following lemma may 
be deduced from the case when $X = \A^d$ (see Example~\ref{ex:t-linear}), and also follows from 
a more general result of Ishii. 

\begin{lemma}\cite{IshArc}\label{Ish}
The cylinders $\{ \Cont^v (D) \mid v \in |\Sigma| \cap N \}$ are precisely the  $J_{\infty} T$-orbits of 
$J_{\infty} X \setminus \bigcup_i J_{\infty} D_i$, and $\Cont^{\ge v}(D) \setminus   \bigcup_i J_{\infty} D_i  = \coprod_{v \leq^{\Sigma} w}
\Cont^w (D)$.   
\end{lemma}

We are now ready to state our geometric interpretation of the equivariant cohomology ring of $X$. 

\begin{corollary}\label{tv}
There is a natural isomorphism $H^*_T X \cong \SR(\Sigma)$ such that the class 
$[\Cont^{ \ge v} (D)] \in H_T^* X$ corresponds to the monomial $x^v \in \SR(\Sigma)$, 
for each lattice point $v \in |\Sigma| \cap N$. 
\end{corollary}

\begin{proof}
It follows from Lemma~\ref{Ish} that $J_{\infty} X \setminus \bigcup_i J_{\infty} D_i = \coprod_{ v \in |\Sigma| \cap N} \Cont^v (D)$ is an equivariant affine paving, and hence Theorem~\ref{t:basis} implies that the classes $\{ [\Cont^{ \ge v} (D)]  \mid v \in |\Sigma| \cap N \}$ form a $\Z$-basis of $H_T^* X$.  Moreover, it follows from Example~\ref{e:ncd} that these classes satisfy the 
multiplication rule \eqref{deformed}:
\begin{equation*}
 [\Cont^{ \ge u} (D)]  \cdot  [\Cont^{ \ge v} (D)]  = \left\{ \begin{array}{ll}
 [\Cont^{ \ge u + v} (D)]  & \textrm{ if } u,v \in \sigma \textrm{ for some } \sigma \in \Sigma, \\
0 & \textrm{ otherwise}.
\end{array} \right.
\end{equation*}
\end{proof}

\begin{remark}\label{r:toric}
In Section~\ref{jet&eq}, we described how one can compare equivariant cohomology rings under proper, birational morphisms. In the toric setting we have the following application:
a proper, birational morphism $f \colon Y(\Delta)  \rightarrow X(\Sigma)$ between smooth toric varieties corresponds to a refinement $\Delta$ of a fan $\Sigma$  in a lattice $N$. Let $\psi$ and $\phi$ denote the piecewise linear functions on $|\Delta| = |\Sigma|$ with value $1$ on the primitive integer vectors of $\Delta$ and $\Sigma$, respectively, and let $E$ and $D$ denote the union of the torus-invariant divisors of $Y$ and $X$, respectively.  We have a bijection between $\Z$-bases of $H^*_T Y$ and $H^*_T X$ such that, for each $v \in |\Delta| \cap N$, 
$\Cont^{v} (E) \subseteq \Cont^{\phi(v) - \psi(v)} (K_{Y/X})$ and
\[
[\Cont^{ \ge v} (E)] \in H_T^{2\psi(v)} Y, \; \; \; [\Cont^{ \ge v} (D)] \in H_T^{2\phi(v)} X. 
\]
\end{remark}

\begin{remark}\label{pretoric}
\emph{Toric prevarieties} are not necessarily separated analogues of toric varieties which first arose in W{\l}odarczyk's work on embeddings of varieties \cite{WloEmbeddings}.  The geometry of a toric prevariety is controlled by an associated multi-fan\footnote{Roughly speaking, a multi-fan is a fan where one does not require two cones to intersect along a common face.}, and we refer the reader to Section 4 in \cite{PayEquivariant} for an introduction to the subject.  The analogue of Corollary~\ref{tv} holds in this case: if $X = X(\Sigma)$ is a smooth $d$-dimensional toric prevariety associated to a multi-fan $\Sigma$ in a lattice $N$ of rank $d$, then the equivariant cohomology ring $H_T^*X$ is isomorphic to the Stanley-Reisner ring of $\Sigma$ \cite{PayEquivariant}.  On the other hand, if $D_1, \ldots, D_r$ denote the $T$-invariant divisors of $X$, the classes 
\[
  \left\{  [\Cont^{\ge \ab} (D_\bullet)] \in H^*_T X \mid \ab =  (a_1,\ldots, a_r) \in \N^r, \; \bigcap_{a_i > 0} D_i \ne \emptyset \right\} 
\]
form a $\Z$-basis of $H_T^*X$, corresponding to a monomial basis of $\SR(\Sigma)$. 
\end{remark}

\begin{remark}
\emph{Hypertoric varieties} may be viewed as a complex-symplectic analogue of toric varieties; 
their geometry is related to the combinatorics of matroids and hyperplane arrangements.  (We refer the reader to \cite{HSToric} and \cite{ProSurvey} for an introduction to the subject.)  A smooth $2d$-dimensional hypertoric variety $Y$ comes with the action of a $d$-dimensional torus $T$, and  
Proudfoot and Webster \cite{PWIntersection} observed that there is an associated smooth toric prevariety $X = X(\Sigma)$ with torus $T$, and a natural $T$-equivariant affine bundle $p: Y \rightarrow X$.  In particular, $H^*_T Y \cong H^*_T X$.  With the notation of Remark~\ref{pretoric}, the classes 
\[
  \left\{ [\Cont^{\ge \ab} (p^{-1}(D_\bullet))] \in H^*_T Y \mid (a_1,\ldots, a_r) \in \N^r, \; \cap_{a_i > 0} \, p^{-1}(D_i) \ne \emptyset \right\} 
\]
therefore form a $\Z$-basis of $H^*_T Y$, corresponding to a monomial basis of $\SR(\Sigma)$. 
\end{remark}

\section{Example: determinantal varieties and $GL_n$}\label{s:gln}

In this section, we apply our results to give a new interpretation of the 
$GL_n$-equivariant cohomology ring of a partial flag variety via contact loci of determinantal varieties.

Consider  $G = GL_n (\C)$ acting by left multiplication on the variety of $n \times n$ matrices $M_{n,n} = M_{n,n}(\C)$. Since $M_{n,n}$ is contractible, Lemma~\ref{l:homotopy} implies that 
$H^*_G M_{n,n} \cong  H^*_G (\pt) = \Lambda_G$. Our first aim is to present a natural, geometric $\Z$-basis for $\Lambda_G$.  Consider the chain of closed subvarieties
\[
  V_n \subseteq \cdots \subseteq V_1 \subseteq V_0 = M_{n,n}, 
\]
where
\[
  V_r = \{ A = (a_{i,j}) \in M_{n,n} \mid \rk (a_{i,j})_{1 \le j \le n + 1 - r } < n + 1 - r \}.
\]
That is, $V_r$ is the subvariety of $M_{n,n}$ defined by setting all $(n+1 -r) \times (n+1 -r)$ minors involving the first $n+1-r$ columns equal to zero.  It is well known that $V_r$ is a normal, irreducible variety of codimension $r$ in $M_{n,n}$.

\begin{remark}
Note that $V_n \cong \A^{n(n - 1)}$ is a smooth ($T$-equivariant) complete intersection, 
and $V_1$ is a singular hypersurface provided $n \ge 2$.  On the other hand, $V_r$ is not a local complete intersection variety for $1 < r < n$. 
\end{remark}


The jet schemes of determinantal varieties have been studied by Musta{\c{t}}{\u{a}}~\cite{MusJet}, 
Yuen~\cite{YueJet}, Ko{\v{s}}ir and Sethuraman~\cite{KSDeterminantal}, and 
Docampo~\cite{DocArcs}.  We will use the following fact:

\begin{theorem}[{\cite[Theorem~3.1]{KSDeterminantal}}]\label{t:determinantal}
The jet schemes $J_m V_r$ are irreducible for all $m \ge 0$ and $1 \le r \le n$.  
\end{theorem}


\begin{remark}
The cases $r = n$ and $r=1$ are easy: $V_n$ is smooth and the result is immediate, while $V_1$ is a normal hypersurface (hence a local complete intersection) with canonical singularities, so the theorem follows from Theorem~\ref{maddog}.  The case $r = n - 1$ is due to Musta{\c{t}}{\u{a}} \cite[Example 4.7]{MusJet}.
\end{remark}

Given a tuple of non-negative integers $\mm = (m_1,\ldots,m_n)$, recall that the partition $\lambda(\mm)=(\lambda_1,\ldots,\lambda_n)$ is defined by $\lambda_i = m_i+\cdots+m_n$ (see \S\S\ref{s:mult1}--\ref{s:mult2}).  Considering the $J_{\infty}G$-invariant cylinders 
\[
\Cont^{\lambda}(V_\bullet) = \Cont^{\lambda(\mm)}(V_\bullet) := \bigcap_{i = 1}^n \Cont^{\lambda_i} (V_i)  \subseteq J_{\infty} M_{n,n}
\]
and
\[
\Cont^{\ge \lambda}(V_\bullet) = \Cont^{\ge \lambda(\mm)}(V_\bullet) := \bigcap_{i = 1}^n \Cont^{\ge \lambda_i} (V_i)  \subseteq J_{\infty} M_{n,n},
\]
observe that
\[
J_\infty M_{n,n}  \setminus J_\infty V_1 = \coprod_{\lambda} \Cont^{\lambda} (V_\bullet),
\]
where $\lambda$ varies over all partitions of length at most $n$.

\begin{lemma}\label{l:gl_contact}
The contact locus $\Cont^\lambda(V_\bullet)$ is an affine family of orbits.
\end{lemma}

\begin{proof}
Identify $J_\infty M_{n,n}$ with $n \times n$ matrices whose entries are power series in 
$\C[[t]]$, and set $m_i = \lambda_i - \lambda_{i + 1}$ for $1 \le i \le n$, so that 
$\lambda = \lambda(\mm)$.  For brevity, we will use the notation
\[
  C = \Cont^\lambda(V_\bullet) \quad \text{ and } C_m = \Cont^\lambda(V_\bullet)_m 
\]
in this proof.

Let $L \subseteq C$ be the set of $n\times n$ 
upper triangular matrices with 
$(i,i)^{\textrm{th}}$ entry equal to $t^{m_{n + 1 - i}}$ and $(i,j)^{\textrm{th}}$ entry equal to 
a polynomial in $t$ of degree strictly less than $m_{n + 1 - j}$ for $i < j$; this is an affine space $\A^N$, for $N = n(m_1+\cdots + m_n) = n\lambda_1$.  Let $L_m \subseteq C_m \subseteq J_m M_{n,n}$ be defined similarly.  Take $m>\lambda_1$, so that $L_m \isom L \isom \A^N$ and $L_m$ is not contained in $J_m V_1$.

Using row operations, one sees that every $J_m G$-orbit in $C_m$ has a unique representative in $L_m$. 
We claim that the map $p\colon C_m \to L_m$ given by
\[
  p(x) = (J_m G \cdot x) \cap L_m
\]
is a smooth, algebraic morphism of varieties.  To see this, consider $x$ as a matrix, and assume it lies in the open subset $U \subseteq C_m$ where the 
top-left minor of size $i$ has order $m_{n+1-i}$ (in $t$).  (By definition, $C_m$ is covered by $n!$ such open sets $U_w$, one for each permutation, since some minor on the first $i$ columns has order $m_{n+1-i}$.)  Thus the entry in position $(1,1)$ has the form $x_{1,1} = t^{m_n}\cdot q(t)$, where $q(t)$ is an invertible element of $\C[t]/(t^{m+1})$.  Scale the $n^\textrm{th}$ row by $q(t)^{-1}$, and use row operations to set the entries below $x_{1,1}$ to zero.  Note that the entries of the resulting matrix $x'$ are rational functions of the coordinates of $x$.  Repeat this process for $x'$, starting with $x'_{2,2}$, with the additional step of using row operations to ensure the entry $x'_{1,2}$ is a polynomial of degree strictly less than $m_{n-1}$.  Continuing in this way, one obtains a matrix in $L_m$ whose entries are rational functions of the coordinates of $x$; that is, we have described a morphism $U\to L_m$.  Here is an example, for $n=2$, $\lambda=(2,1)$, and $m=3$:
\begin{multline*}
x=\left[\begin{array}{ll} t + t^2 & 1 + 2t \\ t & 1 +t^2\end{array}\right]
\leadsto
\left[\begin{array}{ll} t & (1 + 2t)(1 - t + t^2) \\ t & 1 + t^2 \end{array}\right] 
\leadsto
\left[\begin{array}{ll} t & 1  + t - t^2 + 2t^3 \\ 0 & - t + 2t^2 - 2t^3 \end{array}\right] \\
\leadsto
\left[\begin{array}{ll} t & 1  + t - t^2 + 2t^3 \\ 0 &  t \end{array}\right]
\leadsto
\left[\begin{array}{ll}t & 1 \\0 & t\end{array}\right]=p(x).
\end{multline*}

The map is defined similarly on the other open sets $U_w$, by composing with an appropriate permutation of the rows.  Since $p(x)$ is the unique element of $L_m$ in the orbit $J_m G\cdot x$, it follows that these maps patch to give a morphism $C_m \to L_m$.  (In fact, we have described morphisms $s_w\colon U_w \to J_m G$, with $s_w(x)\cdot x = p(x)$.  These maps to $J_m G$ do not glue, however---only the composition with the action map is well defined on the overlaps of the $U_w$'s.)

Finally, consider $\GG = J_m G \times L_m$ as a group scheme over $L_m$, and let $\HH \subseteq \GG$ be the flat subgroup scheme defined by $\HH=\{(g,x) \,|\, g\cdot x = x \}$.  Since the quotient $\GG/\HH = C_m$ exists as a scheme (in fact, a variety), general facts about quotients imply that the maps $\GG \to C_m$ and $C_m \to L_m$ are smooth (see, e.g., \cite[\S I.5]{jantzen}).  The lemma follows.
\end{proof}

Our geometric description of $\Lambda_G$ now follows immediately from Theorem~\ref{t:basis}: 

\begin{corollary}\label{c:glnbasis}
With the notation above, the classes 
$[\overline{\Cont^{\lambda} (V_\bullet)}]$ form a $\Z$-basis of $\Lambda_G$, as $\lambda$ varies over all partitions of length at most $n$.
\end{corollary}


Recall from Example~\ref{ex:gln} that $\Lambda_G  = \Z[c_1, \ldots, c_n]$, where $c_i$ 
is the $i^{\textrm{th}}$ equivariant Chern class of the standard representation 
of $G=GL_n$.  Given a partition $\mu = (\mu_1\geq \cdots \geq \mu_p \geq 0)$ with $\mu_1\leq n$, we also write $c_\mu = c_{\mu_1}\cdot c_{\mu_2} \cdots c_{\mu_p}$.  We offer the following conjecture.

\begin{conjecture}\label{c:gln}
Let $\mm = (m_1,\ldots,m_n)$ be a tuple of non-negative integers, and set $\lambda = \lambda(\mm)$.  Then
\begin{equation}\label{eq:c:gln}
 [\Cont^{\ge \lambda} (V_\bullet)]  = [\overline{\Cont^{\lambda} (V_\bullet)}] = c_1^{m_1} \cdots c_n^{m_n} = c_{\lambda'},
\end{equation}
where $\lambda'$ is the conjugate partition to $\lambda$.
\end{conjecture}

\begin{remark}\label{r:degrees}
It follows from Lemma~\ref{l:gl_contact} that $\Cont^{\lambda} (V_\bullet)$ is a smooth cylinder of codimension $|\lambda|:= \sum \lambda_i$, and hence $\Cont^{\ge \lambda} (V_\bullet)$ and $\overline{\Cont^{\lambda} (V_\bullet)}$ are closed cylinders of codimension $|\lambda|$. 
In particular, the classes in Conjecture~\ref{c:gln} all have the correct degree. 
\end{remark}

\begin{remark}
If $\lambda_1 = \cdots = \lambda_r = m + 1$ and $\lambda_{r + 1} = \cdots = \lambda_n = 0$, then, using Theorem~\ref{t:determinantal},
$\Cont^{\ge \lambda} (V_\bullet)  = \overline{\Cont^{\lambda} (V_\bullet)} =  \psi_{m}^{-1}(J_m V_r)$. 
\end{remark}

\begin{remark}\label{r:known}
We can establish Conjecture~\ref{c:gln} in several cases:
\begin{enumerate}
\item The fact that $[V_r] = c_r$ 
is well known; for example, it follows from the Giambelli-Thom-Porteous formula for 
cohomology classes of degeneracy loci \cite[\S14]{FulIntersection}.

\item Since $V_1$ is a normal e.l.c.i. with rational singularities, it follows from Corollary~\ref{c:affine-CI} that $[V_1]^{m + 1} = [J_m V_1] = c_1^{m + 1}$.

\item Since $V_n$ is smooth, Corollary~\ref{c:mult} says $[J_m V_n] = [V_n]^{m + 1} = c_n^{m + 1}$.  (This is also easy to see directly.) \label{r:Vn}

\item For $m = 1$ and any $r$, Corollary~\ref{c:mult} implies $[J_1 V_r] = [V_r]^2$.  Indeed, for $1\leq r < n$, the singular locus of $V_r$ has codimension $2(r+1)$, so the hypothesis of Corollary~\ref{c:mult} is satisfied when $(m-1)r < 2$.

\item When $n = 2$, the conjecture follows from Theorem~\ref{t:multiplication}, Remark~\ref{r:codimension-closure} and Theorem~\ref{maddog}. 

\item When $n = 3$, we have verified that $[J_m V_2]  = [V_2]^{m + 1} = c_2^{m + 1}$ for $m \leq 5$ using Macaulay 2. 
\end{enumerate}
\end{remark}

Now we use Corollary~\ref{c:birational} to relate the discussion above with partial flag varieties.  
Fix integers $0 = r_0 < r_1 < r_2 <\cdots < r_k < r_{k + 1} = n$, and consider the partial flag variety
\[
  \Fl(\rr) = \Fl(r_1,\ldots,r_k;n) = \{ (V_{r_1} \subseteq \cdots \subseteq V_{r_k} \subseteq \C^n) \mid \dim V_{r_i} = r_i \}.
\]
Let $F_\bullet \in \Fl(\rr)$ be the standard (partial) flag, and let $P$ be the parabolic subgroup of $G$ which fixes $F_\bullet$.  That is, $P$ is the group of invertible block upper-triangular matrices, with diagonal blocks of sizes $r_1, r_2-r_1, \ldots, r_k-r_{k-1}, n-r_k$:

\[
       \begin{array}{r|c|c|c|c|} 
    \multicolumn{1}{c}{}    &  \multicolumn{1}{c}{\overbrace{}^{r_1}} & \multicolumn{1}{c}{\overbrace{}^{r_2-r_1}} &   \multicolumn{1}{c}{}     & \multicolumn{1}{c}{\overbrace{}^{n-r_k}}  \\ 
        \cline{2-5}
 r_1\big\{      &             *        &   *                   & \cdots & *    \\ \cline{2-5}
 r_2-r_1\big\{  &             0        &   *                   & \cdots & *    \\ \cline{2-5}
                &      \vdots          &  \vdots               & \ddots &\vdots\\ \cline{2-5} 
  n-r_k \big\{  &             0        &   0                   & \cdots & *    \\ \cline{2-5}
           \end{array} 
\]
Let $\liep$ be the Lie algebra of $P$; it consists of all matrices with the same block form as $P$.  Note that $P$ acts on $\liep$ by left matrix multiplication, and that $\Fl(\rr)$ is naturally identified with $G/P$.
Consider
\[
  Y = G \times^P \liep,
\]
the quotient of $G\times\liep$ by the relation $(g\cdot p,x)\sim (g,p\cdot x)$ for $p\in P$.  This comes with a $G$-equivariant map $\phi\colon Y \to M_{n,n}$, induced by the multiplication map $G \times \liep \to M_{n,n}$ sending $(g,x)$ to $g\cdot x$.  It is also a vector bundle over $G/P=\Fl(\rr)$ via the first projection, and hence  $H^*_G Y \cong H^*_G \Fl(\rr)$ by Lemma~\ref{l:homotopy}. 
Moreover, 
we have an identification
\[
  Y = S_{r_1}^{\oplus r_1} \oplus S_{r_2}^{\oplus r_2-r_1} \oplus \cdots \oplus S_n^{\oplus n-r_k} \subseteq S_n^{\oplus n} \isom \Fl(\rr)\times M_{n,n},
\]
where $S_r$ is the tautological rank $r$ bundle on $\Fl(\rr)$.  (So $S_n = \C^n$ is the trivial bundle.)    From this perspective, the map $\phi$ is simply projection on the second factor; in particular, $\phi$ is proper.

Recall that $V_{n + 1 - r} \subseteq M_{n,n}$ is the locus of matrices where the first $r$ columns have rank strictly less than $r$.  One sees that $\phi$ is an isomorphism over the open set $M_{n,n}\setminus V_{n + 1 - r_k}$.  Moreover, $E=\phi^{-1}(V_{n + 1 - r_k})$ is a reduced divisor with $k$ irreducible components $E_{n + 1 - r_1},\ldots, E_{n + 1 - r_k}$.  (To see this, lift $\phi$ to the multiplication map $\tilde\phi\colon G\times\liep \to M_{n,n}$, and observe that $\tilde\phi^{-1}(V_{n + 1 - r_k})$ is defined by the vanishing of the principal $r_k\times r_k$ minor in $\liep$.  This determinant factors into $k$ block determinants, of sizes $r_1, r_2-r_1,\ldots r_k-r_{k-1}$.)  \label{page:lift-mult}
In fact, for $1 \le i \le k$, $\phi^{-1}(V_{n + 1 - r_i}) = E_{n + 1 - r_1} +  \cdots +  E_{n + 1 - r_i}$. 

To apply Corollary~\ref{c:birational}, we compute $K_{Y/M_{n,n}}$.  This is equivalent to $K_Y$, since $K_{M_{n,n}}=0$.  In fact, we have
\[
  K_{Y/M_{n,n}} = K_Y = \sum_{i=1}^k (n-r_i) E_{n + 1 - r_i}.
\]
We leave the details of this calculation to the reader; it can be done by considering the vector bundle projection $Y \to \Fl(\rr)$, and using standard formulas for $K_{\Fl(\rr)}$ and the relative canonical divisor of a vector bundle.

For $r$ not among the $r_i$'s, let $\tilde{V}_{n + 1 - r}$ be the ``proper transform'' of $V_{n + 1 - r}$, that is, the closure of $\phi^{-1}(V_{n + 1 - r} \setminus V_{n + 1 - (r - 1))})$; let $\tilde{V}_{n + 1 - r_i} = E_{n + 1 - r_i}$.  If $r_{i-1}<r<r_i$, then $\tilde{V}_{n + 1 - r} \subseteq  E_{n + 1 - r_i}$ and $\phi^{-1}(V_{n + 1 - r}) = \tilde{V}_{n + 1 - r} + E_{n + 1 - r_1} +  \cdots +  E_{n + 1 - r_{i - 1}}$.

Given a partition $\lambda = (\lambda_1 \geq \cdots \geq \lambda_n)$ of length at most $n$, we define a new partition $\tilde{\lambda} = (\tilde\lambda_1, \ldots, \tilde\lambda_n )$ by 
\begin{align*}
  \tilde{\lambda}_{n + 1 - r} &= \begin{cases} \lambda_{n + 1 - r} - \lambda_{n + 1 - r_{i - 1}} &\text{ for } r_{i - 1} < r \le r_i, \\
   \lambda_{n + 1 - r} &\text{ for }r \le r_1. \end{cases}
\end{align*} 
Alternatively, if $\mu$ is the subpartition of $\lambda$ given by 
\[
  \underbrace{\lambda_{n + 1 - r_k} \ge \cdots \ge \lambda_{n + 1 - r_k}}_{n - r_k \textrm{ times}}
\ge \underbrace{\lambda_{n + 1 - r_{k - 1}} \ge \cdots \ge \lambda_{n + 1 - r_{k  - 1}}}_{r_k - r_{k - 1} \textrm{ times}}
\ge \cdots \ge  \underbrace{\lambda_{n + 1 - r_{1}} \ge \cdots \ge \lambda_{n + 1 - r_{1}}}_{r_2 - r_{1} \textrm{ times}},
\]
then $\tilde{\lambda} = \lambda - \mu$. 

Observing that $\phi_{\infty}^{-1} (\Cont^{\lambda}(V_\bullet)) = \Cont^{\tilde{\lambda}} (\tilde{V}_\bullet) \subseteq  \Cont^{e(\lambda)}(K_{Y/M_{n,n}})$, where
\[
  e(\lambda) = \sum_{i = 1}^k (n - r_i)\tilde{\lambda}_{n + 1 - r_i} =  \sum_{i = 1}^k (n - r_i)(\lambda_{n + 1 - r_i} - \lambda_{n + 1 - r_{i - 1}}) = \sum_{i = 1}^n \mu_i ,
\]
we have the following application of Corollary~\ref{c:birational} and Remark~\ref{r:degrees}.

\begin{corollary}\label{c:flag}
With the notation above, the classes 
$[ \overline{\Cont^{\tilde{\lambda}} (\tilde{V}_\bullet)}]$ form a $\Z$-basis of 
$H^*_G \Fl(\rr)$, as $\lambda$ varies over all partitions of length at most $n$.  
Moreover, the degree of $[\overline{\Cont^{\tilde{\lambda}} (\tilde{V}_\bullet)}]$ 
in $H^*_G \Fl(\rr)$ is $|\tilde\lambda|:=\sum_i \tilde{\lambda}_i$. 
\end{corollary}

Note that $H_G^*Y = H_G^*(G/P) = H_P^*(\pt) = \Lambda_P$ is isomorphic to the ring of ``multiply-symmetric functions'' 
\[
  \Lambda_P = \Z[t_1,\ldots,t_n]^{S_{d_0}\times \cdots \times S_{d_k}},
\]
where $d_i = r_{i + 1}-r_{i}$.  One may view this isomorphism as induced from the inclusion 
of $\liep$ into $G \times^P \liep$ sending $A$ to $(1,A)$, which is equivariant with respect 
to the inclusion $P \hookrightarrow G$.  The corollary therefore describes an isomorphism of groups
\begin{equation}\label{eq:g-p}
\begin{array}{rcl}
  \Lambda_G = \Z[t_1,\ldots,t_n]^{S_n} &\to& \Z[t_1,\ldots,t_n]^{S_{d_0}\times \cdots \times S_{d_k}} = \Lambda_P,  \vspace{.1in} \\
   {[\overline{\Cont^{\lambda}(V_\bullet)}]} & \mapsto & {[\overline{\Cont^{\tilde\lambda}(\tilde{V}_\bullet)}]} .
\end{array}
\end{equation}

For $r_i + 1 \le n + 1 - j \le r_{i + 1}$, let $c_{j, \rr} \in \Z[t_1,\ldots,t_n]^{S_{d_0}\times \cdots \times S_{d_k}}$ denote the $(r_{i + 1} - n + j)^{\textrm{th}}$ elementary symmetric 
function in the variables $t_{r_{i} + 1}, \ldots, t_{r_{i + 1}}$.  At the level of symmetric functions, there is an obvious group isomorphism $\Z[t_1,\ldots,t_n]^{S_n} \to \Z[t_1,\ldots,t_n]^{S_{d_0}\times \cdots \times S_{d_k}}$ defined by sending the monomial $c_1^{m_1}\cdots c_n^{m_n}$ to the monomial $c_{1,\rr}^{m_1} \cdots c_{n, \rr}^{m_{n}}$.  We conjecture that this is precisely the bijection defined geometrically in \eqref{eq:g-p}:

\begin{conjecture}\label{c: degeneration2}
If $\mm = (m_1,\ldots,m_n)$ is a tuple of non-negative integers and $\tilde{\lambda} = \tilde{\lambda}(\mm)$, then
\begin{equation}\label{eq:c2}
  [\Cont^{\ge \tilde{\lambda}} (\tilde{V}_\bullet)]  = [ \overline{\Cont^{\tilde{\lambda}} (\tilde{V}_\bullet)}] =  c_{1,\rr}^{m_1} \cdots c_{n, \rr}^{m_{n}}. 
\end{equation}
\end{conjecture}

\begin{example}\label{e:fullflag}
We will show that the conjecture holds in the case of the full flag variety $Fl(n)=G/B$.  
Indeed, in this case $E=\phi^{-1}(V_{1}) = E_{n} + \cdots + E_{1}$ is a simple normal crossings divisor with $n$ irreducible components $E_i = \tilde{V}_i$.  (This is a special case of the remark on p.~\pageref{page:lift-mult}---lift $\phi$ to the multiplication map $\tilde\phi\colon G\times\lieb \to M_{n,n}$, and observe that $\tilde\phi^{-1}(V_{1})$ is defined by the vanishing of the product of the diagonal entries in $\lieb$.)  Moreover, $\tilde{\lambda} = \mm$ and $[E_i] = t_{n + 1 - i}$ under the isomorphism $H_G^* Y \cong H_B^* (\lieb) \cong H_T^* (\pt) = \Z[t_1, \ldots, t_n]$.  The conjecture now follows from Corollary~\ref{c:flag} and Example~\ref{e:ncd}. 
\end{example}

\section{Final remarks}\label{rmks}

It would be interesting to extend the ideas of this paper to the case when $X$ has singularities.  In the case when $X$ has orbifold singularities, we suggest that, on the one hand, one should replace the 
equivariant cohomology $H^*_G X$ with the \define{equivariant orbifold cohomology ring} $H^*_{G, \orb} (X;\Q)$.  Orbifold cohomology was introduced by Chen and Ruan \cite{CRNew} and an algebraic version was developed by Abramovich, Graber and Vistoli \cite{AGVAlgebraic}. One may extend their definitions to define the equivariant version $H^*_{G, \orb} (X;\Q)$.
On the other hand, for $m \in \N \cup \{ \infty \}$, we suggest replacing the jet schemes $J_{m} X$ with the stack of \define{twisted jets} 
$\mathcal{J}_m \mathcal{X}$, as defined by Yasuda \cite{YasMotivic}. 

In the case when $X = X(\Sigma)$ is a simplicial toric variety corresponding to a fan $\Sigma$ in a lattice $N$, one can extend the ideas of Borisov, Chen and Smith \cite{BCSOrbifold} to show that  $H^*_{T, \orb} (X;\Q)$ is isomorphic to the deformed group $\Q[N]^{\Sigma}$ \cite{PSNote}.  
%
On the other hand, an explicit description of the stacks $\mathcal{J}_m \mathcal{X}$ was given by the second author in \cite{StaMotivic}: roughly speaking, away from a closed substack of infinite codimension, the $J_{\infty}T$-orbits of $\mathcal{J}_{\infty} \mathcal{X}$ consist of cylinders $\{ C_v \mid v \in |\Sigma| \cap N \}$.  One expects that under the isomorphism $H^*_{T. \orb} (X;\Q) \cong \Q[N]^{\Sigma}$, the class $[\overline{C_v }]$ in $H^*_{T, \orb} (X;\Q)$ corresponds to $y^v$ in $\Q[N]^{\Sigma}$ for all $v \in |\Sigma| \cap N$.

More generally, we expect that our main results should extend to other situations.  For example, the evidence for Conjecture~\ref{c:gln} suggests that the hypotheses in Theorems~\ref{t:multiplication} and \ref{t:smooth} can be relaxed.  It would also be interesting to study spherical varieties in the spirit of Theorem~\ref{t:basis}, generalizing the example of toric varieties.  

We expect the higher-order equivariant multiplicities defined in \S\ref{s:multiplicities} to have interesting relationships with other singularity invariants.  Focusing on the local case, a natural question is this: do the sequences $\{e^T_{x,m}(V)\}$ and $\{e^T_{x,m}(V)\}$ always have well-defined limits as $m\to\infty$?  It should also be interesting to explore a connection between piecewise polynomials on fans and higher-order multiplicities for singular toric varieties, generalizing the work of Katz and Payne \cite{katz-payne}.

Finally, it would be useful to develop a version of this theory for varieties over an arbitrary field, using equivariant Chow groups.  The statements of our results make sense in this context, so we expect this should be possible; however, there are a few technical obstacles, since several of our proofs use analytic neighborhoods and the long exact sequence for Borel-Moore homology.

\bibliographystyle{amsplain}
\def\cprime{$'$}
\providecommand{\bysame}{\leavevmode\hbox to3em{\hrulefill}\thinspace}
\providecommand{\MR}{\relax\ifhmode\unskip\space\fi MR }
\providecommand{\MRhref}[2]{%
  \href{http://www.ams.org/mathscinet-getitem?mr=#1}{#2}
}
\providecommand{\href}[2]{#2}
%


\end{document}